\documentclass[a4paper,10pt,reqno]{amsart}
          \usepackage{amssymb}
	  \usepackage{amsmath}
          \usepackage{amsfonts}
          \usepackage[english]{babel}
          \usepackage[utf8]{inputenc}

\usepackage[margin=3cm]{geometry}

\usepackage{color}
 \usepackage[unicode,colorlinks,plainpages=false,hyperindex=true,bookmarksnumbered=true,bookmarksopen=true,pdfpagelabels]{hyperref}
 \hypersetup{urlcolor=cyan,linkcolor=blue,citecolor=red,colorlinks=true} 
\makeatletter
\renewcommand*{\p@section}{\S\,}
\renewcommand*{\p@subsection}{\S\,}
\renewcommand*{\p@subsubsection}{\S\,}
\makeatother

\usepackage{tikz}
 \usetikzlibrary{arrows}

\newtheorem{thm}{Theorem}[section]

\newtheorem{lem}[thm]{Lemma}
\newtheorem{prop}[thm]{Proposition}

\newtheorem{rem}[thm]{Remark}

\theoremstyle{definition}

\numberwithin{equation}{section}

\newcommand{\CC}{\ensuremath{\mathbb{C}}}
\newcommand{\N}{\ensuremath{\mathbb{N}}}

\newcommand{\Z}{\ensuremath{\mathbb{Z}}}
\newcommand{\PP}{\operatorname{P}}

\newcommand{\tr}{\operatorname{tr}}

\newcommand{\diag}{\operatorname{diag}}

\newcommand{\Hom}{\operatorname{Hom}}
\newcommand{\Der}{\operatorname{Der}}
\newcommand{\Mat}{\operatorname{Mat}}
\newcommand{\Id}{\operatorname{Id}}
\newcommand{\Gl}{\operatorname{GL}}

\newcommand{\Rep}{\operatorname{Rep}}
\newcommand\dgal[1]{  \left\{\!\!\left\{#1\right\}\!\!\right\} }
\newcommand\dSN[1]{\left\{\!\!\left\{#1\right\}\!\!\right\}_{\operatorname{SN}}}
 
\def\psil{\psi_{\ell}}
\def\nn{\mathcal N}
\def\aalpha{{\widetilde{\alpha}}}
\def\qq{{\widetilde{q}}}

\def\h{\mathfrak{h}}
\def\hreg{\mathfrak{h}_{\mathrm{reg}}}

\def\VV{\mathcal V}

\begin{document}

\title[Quivers and RS models] 
{Multiplicative quiver varieties and generalised Ruijsenaars-Schneider models}

%
\author{Oleg Chalykh}
 \address[Oleg Chalykh]{School of Mathematics\\
         University of Leeds\\
         Leeds, LS2 9JT, UK\\}
 \email{o.chalykh@leeds.ac.uk}

\author{Maxime Fairon}
 \address[Maxime Fairon]{School of Mathematics\\
         University of Leeds\\
         Leeds, LS2 9JT, UK\\}
 \email{mmmfai@leeds.ac.uk}


\begin{abstract}
We study some classical integrable systems naturally associated with multiplicative quiver varieties for the (extended) cyclic quiver with $m$ vertices. The phase space of our integrable systems is obtained by quasi-Hamiltonian reduction from the space of representations of the quiver. Three families of Poisson-commuting functions are constructed and written explicitly in suitable Darboux coordinates. The case $m=1$ corresponds to the tadpole quiver and the Ruijsenaars--Schneider system and its variants, while for $m>1$ we obtain new integrable systems that generalise the Ruijsenaars--Schneider system. These systems and their quantum versions also appeared recently in the context of supersymmetric gauge theory and cyclotomic DAHAs \cite{BEF, BFNa, BFNb, KN}, 
as well as in the context of the Macdonald theory \cite{CE}.
\end{abstract}

\maketitle

 \setcounter{tocdepth}{1}


\section{Introduction}
\label{intro}

Among the most powerful geometric techniques in the theory of integrable systems is the method of Hamlitonian (or symplectic) reduction. Invented initially for reducing the degrees of freedom in Hamiltonian systems with symmetries, the concept of a moment map and symplectic reduction have since found a multitude of uses beyond their initial scope. One of the earlier examples of a Hamiltonian reduction was given by Kazhdan, Kostant and Sternberg \cite{KKS78}, who demonstrated how to obtain the celebrated Calogero--Moser system from a very simple system on $T^*\mathfrak{gl}_n$. Since then, many integrable systems have been obtained or interpreted by similar methods. Among those is a remarkable generalisation of the Calogero--Moser system introduced by Ruijsenaars and Schneider \cite{Ruijs86}. The latter system was interpreted in terms of an infinite-dimensional symplectic reduction by Nekrasov \cite{N}, extending his earlier work with A.~Gorsky \cite{GN}. Hyperbolic Ruijsenaars--Schneider system can also be obtained by a finite-dimensional reduction in the spirit of \cite{KKS78}, as was demonstrated by Fock and Rosly \cite{FockRosly}. Although Fock and Rosly employ Hamiltonian (or Poisson) reduction, their construction allows an interpretation in terms of \emph{quasi-Hamiltonian} reduction. We recall that the method of quasi--Hamiltonian reduction was developed by Alekseev, Malkin and Meinrenken in \cite{AMM}, see also \cite{QuasiP}.  The main difference is that the reduction is performed on a space which may not be symplectic, and the moment map takes values in the Lie group rather than the Lie algebra. Not attempting at a comprehensive review, we refer the reader to some of the more recent papers \cite{P11, FK09a, FK09b, FehKlim12, M, FG}, where the Ruijsenaars--Schneider model and its variants are treated by the method of (quasi-)Hamiltonian reduction, and where further  references can be found. Let us also mention an alternative geometric approach to many-body problems by Krichever \cite{Kr}, in which the Lax matrix structure play the central role instead, and the Hamiltonian picture is derived from that, cf. \cite{KrP, Kr2, KrS}.

From yet another perspective, a unified view onto the Calogero--Moser and Ruijsenaars--Schneider system can be achieved by noticing that in both cases the reduction is done on (the cotangent bundle to) the space of representations of a one-loop quiver. Such a view onto the (complexified) Calogero--Moser system was brought forward by G.~Wilson's work \cite{W} relating the rational Calogero--Moser system, adelic Grassmannian, and the KP hierarchy, and it has been further deepened in \cite{BW1}, \cite{BW2}, cf. \cite{BC07, BP}. 
The present paper stems from a natural idea to look for a generalisation of Wilson's results for more complicated quivers. We recall that with any quiver Nakajima associates in \cite{Nak94} a class of symplectic quotients called \emph{quiver varieties}. There exists also a multiplicative version of quiver varieties, introduced by Crawley-Boevey and Shaw \cite{CBShaw} and interpreted via quasi-Hamiltonian reduction by Van den Bergh \cite{VanDenBergh}. Affiine Dynkin quivers are a particularly well studied class, and a large part of Wilson's (and Berest--Wilson's) results have already been extended to this case by  Ginzburg, Baranovsky and Kuznetsov \cite{BGK1}, \cite{BGK2} (see also \cite{Eshmatov}, \cite{CS}). However, the multiplicative case has not been systematically looked at, apart from the already mentioned case of a one-loop quiver.  
This was the main motivation behind our work. We will focus on the link to integrable particle dynamics; other aspects of the Calogero--Moser correspondence will be discussed elsewhere. Our main result is a construction of new generalisations of the Ruijsenaars--Schneider system, related to cyclic quivers. This is achieved by performing a quasi-Hamiltonian reduction on the space of representations of the associated multiplicative preprojective algebras $\Lambda^q$ of Crawley-Boevey and Shaw \cite{CBShaw}. Our main tool is the formalism of double (quasi-)Poisson algebras due to Van den Bergh \cite{VanDenBergh, VdBquasiSpaces}. The constructed integrable systems come equipped with a complete phase space (represented by a suitable multiplicative quiver variety), and the associated Hamiltonian dynamics can be explicitly integrated. By constructing Darboux coordinates on the phase space, we express the new integrable Hamiltonians in coordinates, which then allows us to identify them as generalisations of the Ruijsenaars--Schneider system.

Interestingly, quantum versions of some of these systems appeared in a seemingly unrelated context in \cite{CE}, where they were called \emph{twisted Macdonald--Ruijsenaars systems}. By computing some of these quantum Hamiltonians explicitly, we are able to see this relationship in the case of the quiver with two vertices, although a direct comparison in the general case is more difficult. However, a recent work by Braverman, Etingof and Finkelberg \cite{BEF}, which appeared while we were finishing the present paper, clarifies this connection rather remarkably. It introduces a cyclotomic version of the double affine Hecke algebra (DAHA) in type $A$. Inside the cyclotomic DAHA there are three natural commutative subalgebras and they give rise to quantum integrable systems, in the same way as the usual DAHA can be used to produce the Macdonald--Ruijsenaars operators. The classical versions of these systems correspond to the $q=1$ limit of the cyclotomic DAHA (cf. \cite{Oblomkov} for the case of the usual DAHA), and this leads to the multiplicative quiver varieties for the cyclic quiver. Thus, the integrable systems constructed in \cite{BEF} coincide (on the classical level) with those constructed by us. The interpretation of these integrbale systems via the cyclotomic DAHA in \cite{BEF} allows to explain their relationship to the twisted Macodnald--Ruijsenaars systems from \cite{CE} in type $A$. Our methods are quite different in comparison, and they allow us to find explicit formulas for the corresponding classical Hamiltonians and integrate the Hamiltonian flows (the approach via the cyclotomic DAHA in \cite{BEF} is less explicit). Curiously, these Hamiltonians become much simpler under the Cherednik--Fourier transform. In this form they appeared in the work of Braverman, Finkelberg, and Nakajima \cite{BFNa, BFNb} on the quantized Coulomb branch of quiver gauge theories, see also a related work of Kodera and Nakajima \cite{KN}. This can also be seen from our formulas at the classical level, when the Cheredink--Fourier transform becomes the angle-action transform studied by Ruijsenaars \cite{R88}. See Section \ref{final} below for more details. Apart from being more explicit compared to \cite{BEF}, our approach also has an advantage of being better suited for studying spin versions of the Ruisjenaars-Scheider system and its generalisations; this will be a subject of a future work.  

The structure of the paper is as follows. In Section \ref{prel} we first describe the general formalism of double Poisson brackets and quasi-Poisson algebras due to Van den Bergh \cite{VanDenBergh}, and then exemplify it for the multiplicative quiver varieties. Section \ref{tadpole} looks at the tadpole quiver, explaining how to obtain the hyperbolic Ruijsenaars--Schneider system by quasi-Hamiltonian reduction. 
In Section \ref{cyclic} we consider the multiplicative quiver varieties (Calogero--Moser spaces) for the framed cyclic quiver with $m$ vertices. We introduce three Poisson commuting families of functions on those Calogero--Moser spaces, and integrate the corresponding Hamiltonian flows. We then write these Hamiltonians in suitable Darboux coordinates, identifying them as generalisations of the hyperbolic Ruijsenaars--Schneider system. Finally, in Section \ref{final} we discuss the relationship between our work and the results of \cite{CE} and \cite{BFNa, BFNb, KN, BEF}. In particular, we are able to write explicitly the integrable quantum Hamiltonaians from \cite{BEF} in the case of the quiver with two vertices. The paper finishes with three appendices containing some of the more technical proofs.

{\it Acknowledgement.} It is our pleasure to thank L. Feh{\'e}r, J.F. van Diejen, P. Etingof, P. Iliev, I. Marshall, S. Ruijsenaars for useful comments and discussions. We are especially grateful to P. Etingof who shared with us the results of \cite{BEF} before their publication. This allowed us to streamline some of our considerations in Section \ref{cyclic}. A part of this paper was written during the first author's stay in November 2016 at the Centro Internacional de Ciencias (CIC), Cuernavaca. He thanks the organisers of the programme "Integrable and quasi-integrable systems" and CIC for the hospitality and excellent working conditions. 
The first author was partially supported by EPSRC under grant EP/K004999/1.  
The second author is supported by a University of Leeds 110 Anniversary Research Scholarship.

\section{Preliminaries}  \label{prel}

In this section we first recap the theory of double Poisson brackets and double (quasi-)Poisson algebras due to Van den Bergh \cite{VanDenBergh}. We then describe a concrete example of this formalism, related to multiplicative preprojective algebras and multiplicative quiver varieties of Crawley-Boevey and Shaw \cite{CBShaw}. We will follow the notation of the papers \cite{VanDenBergh, VdBquasiSpaces}, where the reader can find many more details. 

\subsection{Double brackets and double derivations} \label{ss:dAS}
We fix an algebra $A$ over $\CC$. 
For an element $a\in {A}\otimes {A}$, we will use a shorthand notation $a'\otimes a''$ for $\sum_i a_i'\otimes a_i''$. We set $a^\circ=a''\otimes a'$. More generally,  
for any $s\in S_n$ we define $\tau_s:\,A^{\otimes n}\to A^{\otimes n}$ by 
$\tau_s(a_1\otimes \ldots \otimes a_n)=a_{s^{-1}(1)}\otimes \ldots \otimes a_{s^{-1}(n)}$. Thus, $a^\circ=\tau_{(12)}a$. The multiplication map $m:A^{\otimes n} \to A$ is 
$m(a_1\otimes \ldots \otimes a_n)=a_1\ldots a_n$. We view $A^{\otimes n}$ as an $A$-bimodule via the \emph{outer} bimodule structure 
$b(a_1\otimes \ldots \otimes a_n)c=ba_1\otimes \ldots \otimes a_nc$. 
If $B$ is a $\CC$-algebra, then we say that $A$ is a $B$-algebra if it is equipped with a $\CC$-algebra 
map $B\to A$. 

An \emph{$n$-bracket} is a linear map 
$\dgal{-,\ldots,-} : A^{\otimes n}\to A^{\otimes n}$ which is a derivation in its last 
argument for the outer bimodule structure on $A^{\otimes n}$, 
and which is cyclically anti-symmetric: 
\begin{equation*}
\tau_{(1\ldots n)}\circ \dgal{-,\ldots,-}\circ \tau^{-1}_{(1\ldots n)}
=(-1)^{n+1}\dgal{-,\ldots,-}\,. 
\end{equation*}
If $A$ is a $B$-algebra then we assume that the bracket is $B$-linear, i.e. it vanishes when its last argument is in the image of $B$.  
We call a $2$- and a $3$-bracket respectively a \emph{double} and a \emph{triple bracket}. 
In particular, a double bracket satisfies $\dgal{a,b}=-\dgal{b,a}^\circ$ and 
$\dgal{a,bc}=b\dgal{a,c}+\dgal{a,b}c$. Any double bracket $\dgal{-,-}$  defines an induced triple bracket 
$\dgal{-,-,-}$ given by  
\begin{equation*}
\label{Eq:TripBr}
 \dgal{a,b,c}=\dgal{a,\dgal{b,c}'}\otimes \dgal{b,c}''+\tau_{(123)}\dgal{b,\dgal{c,a}'}\otimes \dgal{c,a}''
+\tau_{(123)}^2\dgal{c,\dgal{a,b}'}\otimes \dgal{a,b}'' \,.
\end{equation*}

A double bracket on $A$ is called a \emph{double Poisson} bracket if the associated triple bracket vanishes. 
For any double (respectively, double Poisson) bracket $\dgal{-,-}$, the bracket $\{-,-\}:=m \circ \dgal{-,-}$ descends to an antisymmetric biderivation (respectively, a Lie bracket) on $A/[A,A]$ \cite[2.4.1, 2.4.6]{VanDenBergh}.

Following \cite{CB}, we call the elements of $D_{A/B}:=\Der_B(A,A\otimes A)$ \emph{double derivations}, and we make $D_{A/B}$ into an $A$-bimodule by using the \emph{inner} bimodule structure on $A\otimes A$:  
if $\delta\in D_{A/B}$ and $a,b,c\in A$, then $b\,\delta\, c (a)=\delta(a)'\,c\otimes b\,\delta(a)''$. Let $D_BA:=T_AD_{A/B}$ be the tensor algebra of this bimodule; this is a graded algebra, with $A$ placed in degree $0$ and $D_{A/B}$ in degree $1$. The elements of degree $n$ in $D_BA$ can be used to define $n$-brackets on $A$: 
\begin{prop}   \emph{(\cite[4.1.1]{VanDenBergh})}
\label{Prop:BrQ}
There is a well-defined linear map $\mu:(D_BA)_n\to \{B$-linear $n$-brackets on $A\}$, 
$Q\mapsto \dgal{-,\ldots,-}_Q$ which on $Q=\delta_1 \ldots \delta_n$ is given by 
\begin{equation*}
  \begin{aligned}
\label{Eq:BrQ}
\dgal{-,\ldots,-}_Q&=\sum_{i=0}^{n-1}(-1)^{(n-1)i}\tau^i_{(1\ldots n)}\circ\dgal{-,\ldots,-}_Q^{\widetilde{ }}
\circ\tau^{-i}_{(1\ldots n)}\,\,, \\
\dgal{a_1,\ldots,a_n}_Q^{\widetilde{ }}&=\delta_n(a_n)'\delta_1(a_1)''\otimes \delta_1(a_1)'\delta_2(a_2)''\otimes \ldots \otimes 
\delta_{n-1}(a_{n-1})'\delta_n(a_n)''\,.
 \end{aligned}
\end{equation*}
The map $\mu$ factors through $D_BA/[D_BA,D_BA]$ (for the graded commutator).
\end{prop}
The algebra $D_BA$ may be viewed as a noncommutative version of the algebra of polyvector fields: according to \cite[3.2.2]{VanDenBergh} $D_BA$ admits a canonical \emph{double Schouten--Nijenhuis bracket}, which makes $D_BA$ into a double Gerstenhaber algebra. We denote this bracket as $\dSN{-,-}$, and we compose it with the multiplication on $D_BA$ to obtain $\{-,-\}_{\operatorname{SN}}:=m\circ\dSN{-,-}$.

\subsection{Double quasi-Poisson algebras} \label{ss:doubleqP}
We now assume that $B$ is commutative and semi-simple, \emph{i.e.}  
$B={\CC} e_1 \oplus \ldots \oplus {\CC} e_n$ with $e_ie_j=\delta_{ij}e_i$. 
We define for all $i$ a double derivation $E_i\in D_{A/B}$ such that 
$ E_i(a)=ae_i\otimes e_i - e_i\otimes ae_i$. 
A \emph{double quasi-Poisson bracket} on $A$ is a $B$-linear bracket $\dgal{-,-}$, 
such that the induced triple bracket satisfies $\dgal{-,-,-}=\frac{1}{12}\sum_i \dgal{-,-,-}_{E_i^3}$, where the brackets in the right-hand side are defined in Proposition \ref{Prop:BrQ}. 
In this case, we say that $A$ is a \emph{double quasi-Poisson algebra}.

A \emph{multiplicative moment map} for a double quasi-Poisson algebra $(A,\dgal{-,-})$ is an element 
$\Phi=\sum_i\Phi_i$ with $\Phi_i\in e_iAe_i$ such that we have   
$\dgal{\Phi_i,a}=\frac{1}{2}(\Phi_iE_i+E_i\Phi_i)(a)$ for all $a\in A$. When a double quasi-Poisson algebra 
is equipped with a multiplicative moment map, we say that it is a \emph{quasi-Hamiltonian algebra}. 

Assume that there is an element $P\in (D_BA)_2$ such that 
$\{P,P\}_{\operatorname{SN}}=\frac{1}{6}\sum_iE_i^3$ mod $[D_BA,D_BA]$ (for the graded commutator). Then 
we say that $A$ is a \emph{differential double quasi-Poisson algebra} with the 
\emph{differential double quasi-Poisson bracket} $\dgal{-,-}_P$. This implies that $\dgal{-,-}_P$ is a double quasi-Poisson 
bracket \cite[4.2.3]{VanDenBergh}. 

\subsection{Representation spaces} \label{ss:PoissonVar}
For a $\CC$-algebra $A$ and any $N\in\N$, the representation space $\Rep(A, N)$ is the affine scheme that parametrises algebra homomorphisms $\varrho: A \to \Mat_N(\CC)$. The coordinate ring $\mathcal{O}(\Rep(A, N))$ is generated by the functions $a_{ij}$ for $a\in A$, $i, j=1, \ldots, N$ defined by $a_{ij}(\varrho)=\varrho(a)_{ij}$ at any point $\varrho\in \Rep(A, N)$. The functions $a_{ij}$ are linear in $a$ and satisfy the relations $(ab)_{ij}=\sum_k a_{ik}b_{kj}$. We can therefore associate with any $a\in A$ a \emph{matrix-valued} function $\mathcal{X}(a):=(a_{ij})_{i,j=1,\ldots, N}$ on $\Rep(A,N)$. Similarly, any double derivation $\delta\in\Der (A,A\otimes A)$ gives rise to a matrix-valued vector field $\mathcal{X}(\delta)=(\delta_{ij})_{i,j=1,\ldots, N}$ on $\Rep(A, N)$, where $\delta_{ij}$ is a derivation of $\mathcal{O}(\Rep(A, N))$ defined by the rule $\delta_{ij}(a_{uv})=\delta'(a)_{uj} \delta''(a)_{iv}$. 

Everything can also be defined in a relative setting, i.e. for a $B$-algebra $A$, where $B$ is of the form $B={\CC} e_1\oplus\ldots\oplus{\CC} e_n$ with $e_i^2=e_i$. Representation spaces are now indexed by $n$-tuples $\alpha=(\alpha_1, \ldots, \alpha_n)\in \N^n$. Given $\alpha$ with $\alpha_1+\ldots +\alpha_n=N$, we embed $B$ diagonally into $\Mat_N(\CC)$ so that $\Id_N$ is split into a sum of $n$ diagonal blocks of size $\alpha_1, \ldots, \alpha_n$, representing the idempotents $e_i$. By definition, $\Rep_B(A, \alpha)=\Hom_B(A,\Mat_{N}(\CC))$, and it can be viewed as an affine scheme in the same way as $\Rep(A, N)$. Note in particular that for any $\Phi\in\oplus_i \, e_i A e_i$, the matrix-valued function $\mathcal{X}(\Phi)$ on $\Rep_B(A,\alpha)$ takes values in block matrices $\prod_i\Mat_{\alpha_i}({\CC})$.

On $\Rep(A, N)$ we have a natural action of $\Gl_N$, induced by conjugation action on $\Mat_N(\CC)$. Similarly, we have an action of  $\Gl_\alpha=\prod_i \Gl_{\alpha_i}$ on $\Rep_B(A, \alpha)$. 

\subsection{Quasi-Poisson manifolds}\label{QPM}
A double quasi-Poisson bracket on an algebra $A$ makes its representation space into a quasi-Poisson manofld. Let us first recall the geometric setup of \cite{ QuasiP}, following the notation of \cite[7.13]{VanDenBergh}. Let $M$ be a $G$-manifold, for $G$ a Lie group whose Lie algebra $\mathfrak{g}$ admits a non-degenerate $G$-invariant bilinear form $(-,-)$. If $(e_a)$ is a basis of $\mathfrak{g}$ and $(e^a)$ the dual basis with respect to 
$(-,-)$, the Cartan $3$-tensor is defined as $\phi=\frac{1}{12}C^{abc}e_a\wedge e_b \wedge e_c$, for 
$C^{abc}=(e^a,[e^b,e^c])$ the tensor of structure constants. Write $\xi^L$ and $\xi^R$ respectively to denote the 
left and right invariant vector fields on $G$, for $\xi \in \mathfrak{g}$. 

The $G$-action on $M$ gives rise to a Lie algebra homomorphism 
$(-)_M:\mathfrak{g}\to \Der \mathcal O(M)$, which can be extended to polyvector fields to define a $3$-tensor $\phi_M$.  
We say that $M$ is a quasi-Poisson manifold if there exists an invariant bivector field $P$ on $M$ such that its 
Schouten-Nijenhuis bracket with itself satisfies $[P,P]=\phi_M$. One associates with $P$ a bracket on $\mathcal O(M)$ defined by $\{f,g\}=P(\mathrm{d}f, \mathrm{d}g)$.
A \emph{multiplicative moment map} is an $\operatorname{Ad}$-equivariant 
map $\Phi:M\to G$  satisfying 
\begin{equation}
\{g\circ \Phi, -\}=\frac{1}{2}(e^a)_M\left((e_a^L+e_a^R)(g)\circ \Phi\right)\,, 
\end{equation}
for all functions $g\in\mathcal O(G)$.
A \emph{Hamiltonian quasi-Poisson manifold} is such a triple $(M, P, \Phi)$. In the case where the action of $G$ on $M$ is free and proper, for each conjugacy class $\mathcal{C}_g$ 
of $g\in G$ the subset $\Phi^{-1}(\mathcal{C}_g)/G$ is a Poisson manifold, and this process is called 
\emph{quasi-Hamiltonian reduction}. 

Now let us turn to geometric structures induced on representation spaces of a double quasi-Poisson algebra $A$. 
Assume that $\dgal{-,-}:A\times A\to A\otimes A$ is a $B$-linear double bracket on $A$.

\begin{prop} \label{Prop:BrInduced} \emph{(\cite[7.5.1, 7.5.2, 7.8, 7.12.2]{VanDenBergh})} 
There is a unique antisymmetric biderivation (bivector)
$\{-,-\}:\mathcal{O}(\Rep_B(A,\alpha))\times \mathcal{O}(\Rep_B(A,\alpha)) \to \mathcal{O}(\Rep_B(A,\alpha))$
such that for all $a,b\in A$, 
\begin{equation}\label{derr}
 \{a_{ij},b_{uv}\}=\dgal{a,b}'_{uj}\, \dgal{a,b}''_{iv}\,\, . 
\end{equation}
If $\dgal{-,-}$ is a double Poisson bracket, then this bivector is Poisson so $\mathcal O(\Rep_B(A, \alpha))$ is a Poisson algebra.
\end{prop}

\begin{thm}
\label{Thm:Poiss}  \cite[7.8, 7.13.2]{VanDenBergh}
Assume that $(A,P)$ is a differential double quasi-Poisson algebra, which is quasi-Hamiltonian for the 
multiplicative moment map $\Phi\in \oplus_ie_iAe_i$. We have that $\Rep_B(A, \alpha)$ is a $\Gl_\alpha$-space with a quasi-Poisson bracket $\{-,-\}$ determined from $\dgal{-,-}_P$ by \eqref{derr}.
Then the matrix-valued function $\mathcal{X}(\Phi):\Rep_B(A,\alpha)\to\prod_i\Mat_{\alpha_i}({k})$ is a (geometric) multiplicative moment map 
for $\Rep_B(A,\alpha)$. Therefore, $\Rep_B(A, \alpha)$ (if smooth) admits a structure of a Hamiltonian quasi-Poisson manifold.  
\end{thm}

\subsection{Multiplicative preprojective algebras} \label{ss:MultqVar}

Let $Q=(Q, I)$ be a quiver with vertex set $I$ and arrow set $Q$. Let $\bar{Q}$ denote the double of $Q$, obtained by adjoining to every arrow $a\in Q$ its 
opposite, $a^{*}$. Define the maps $t,h:\bar{Q}\to I$ that associate to every arrow $a$ its tail and head, $t(a)$ and $h(a)$. In particular, $t(a)=h(a^*)$ and $h(a)=t(a^*)$.  
We define $\epsilon : \bar{Q}\to \{\pm 1\}$ the sign function which associates $1$ to every 
arrow of $Q$ and $-1$ to each arrow of $\bar{Q} \setminus Q$. We write $\CC\bar{Q}$ for the path algebra of 
$\bar{Q}$; it is generated by the idempotents $e_i$ (zero paths) associated to the vertices $i\in I$, and arrows $a\in\bar{Q}$, with multiplication given by concatenation of paths. We view $\CC\bar{Q}$ as a $B$-algebra, with $B=\oplus_{i\in I} \CC e_i$. Finally, we extend $*$ to an involution on $\CC\bar{Q}$ by setting $(a^*)^*=a$.

\begin{rem}\label{rem2.1}
We will use the convention that the composition of paths in $\CC\bar{Q}$ is written from left to right, i.e., $ab$ means ``$a$ followed by $b$", with $ab=0$ if $h(a)\ne t(b)$. 
\end{rem}

Let $A$ be obtained from $\CC \bar{Q}$ by inverting all elements $(1+aa^*)_{a\in \bar{Q}}$.  For all $a\in \bar{Q}$, 
define the element $\frac{\partial}{\partial a}$ 
of $D_BA$ which on $b\in \bar{Q}$ acts as 
\begin{equation}
\label{Eq:Dba}
 \frac{\partial b}{\partial a}=\left\{\begin{array}{ll} e_{t(a)}\otimes e_{h(a)}&\text{if }a=b \\ 
0&\text{otherwise} \end{array}\right. 
\end{equation}
Fix an arbitrary ordering $<$ on $\bar{Q}$ and consider the following element $\Phi\in A$:
\begin{equation}\label{pphi}
\Phi=\prod_{a\in \bar{Q}} (1+aa^*)^{\epsilon(a)}\,,
\end{equation}
where the product is taken with respect to the chosen ordering $<$. Following \cite{CBShaw}, given $q=\sum_{i\in I} q_ie_i$ with $q_i\in \CC^\times$, we define the deformed \emph{multiplicative preprojective algebra} as the quotient $\Lambda^q=A/(\Phi-q)$. Up to isomorphism, the algebra $\Lambda^q$ is independent of the ordering \cite[Theorem 1.4]{CBShaw}. 

\begin{thm} \emph{(\cite[6.7.1.]{VanDenBergh})} 
\label{Thm:QStruct}
 The algebra $A$ has a quasi-Hamiltonian structure given by 
\begin{equation}\label{Eq:PP}
 \PP=\frac{1}{2}\sum_{a\in \bar{Q}} \epsilon(a) (1+a^*a)\frac{\partial}{\partial a} 
\frac{\partial}{\partial a^*} \, - \,\frac12 \sum_{a<b\in \bar{Q}} \left(\frac{\partial}{\partial a^*}a^*-a
\frac{\partial}{\partial a}\right)\left(\frac{\partial}{\partial b^*}b^*-b\frac{\partial}{\partial b}\right) 
\,,
\end{equation}
and the multiplicative moment map given by \eqref{pphi}. 
\end{thm}
Note that the corresponding double-Poisson bracket is defined also on $\CC\bar{Q}$, but the elements $1+aa^*$ need to be invertible to define the moment map $\Phi$. 

The corresponding double quasi-Poisson bracket on $\CC\bar{Q}$ is calculated as follows.  

\begin{prop}\label{Pr:dbr} Suppose that the arrows of $\bar{Q}$ are ordered in such a way that $a<a^*<b<b^*$ for any $a<b\in Q$. Let $\dgal{-,-}$ be the double bracket associated to the bivector \eqref{Eq:PP}. 
Then one has:
 \begin{subequations}
       \begin{align}
\dgal{a,a}\,=\,&\frac{1}{2}\epsilon(a)\left( a^2\otimes e_{t(a)}- e_{h(a)}\otimes a^2 \right) \qquad (a\in\bar{Q})\,, \label{loop}\\
\dgal{a,a^*}\,=\,&e_{h(a)}\otimes e_{t(a)}
+\frac{1}{2} a^*a\otimes e_{t(a)} \nonumber\\
&+\frac{1}{2} e_{h(a)}\otimes aa^*
+\frac{1}{2} (a^*\otimes a-a\otimes a^*)\delta_{h(a),t(a)}\qquad (a\in Q)\,, \label{aast}\\
\dgal{a,b}\,=\,&\frac{1}{2}e_{h(a)}\otimes ab +\frac{1}{2} ba\otimes e_{t(a)} \nonumber\\\label{a<b}
&-\frac{1}{2}(b\otimes a) \delta_{h(a),h(b)}-\frac{1}{2}(a\otimes b) \delta_{t(a),t(b)}\qquad (a,b\in\bar{Q}\,, \ a<b\,,\ b\ne a^*)\,.
\end{align}
  \end{subequations}
\end{prop} 
All other brackets are obtained by using $\dgal{a,b}=-\dgal{b,a}^\circ$, with 
$(c'\otimes c'')^\circ=c''\otimes c'$. Note that \eqref{loop} is zero unless $t(a)=h(a)$ (i.e. $a$ is a loop), and \eqref{a<b} is zero unless $a$ and $b$ share a vertex. 

\medskip

\begin{proof}
We give a proof for \eqref{loop} with $a\in Q$, other formulas are similar. First, recall that if $\delta_1,\delta_2\in D_{A/B}$, then the double bracket associated with $\delta_1\delta_2\in (D_BA)_2$ is given in Proposition \ref{Prop:BrQ}:
\begin{equation} \label{Eq:dgal}
\dgal{a, b}_{\delta_1 \delta_2}=\delta_2(b)'\delta_1(a)''\otimes \delta_1(a)'\delta_2(b)''-\delta_1(b)'\delta_2(a)''\otimes \delta_2(a)'\delta_1(b)''\,.
\end{equation}
When calculating $\dgal{a,a}_{\PP}$, the only nonzero contribution comes from the bivector
\begin{equation*}
 \PP_{a}=\frac{1}{2}\left(a\frac{\partial}{\partial a}\right)\left(\frac{\partial}{\partial a}a\right)\,, 
\end{equation*}
which we write as
\begin{equation}
 \PP_{a}=\frac{1}{2}U^-U^+\,,\qquad U^+=\frac{\partial}{\partial a}a, \quad U^-=a\frac{\partial}{\partial a}\,.\label{Eq:Pdd}
\end{equation}
Using the inner bimodule structure on $D_{A/B}$, we get 
\begin{equation} \label{Eq:UU}
  \begin{aligned}
U^+(a)&=
e_{t(a)}a\otimes e_{h(a)}=a\otimes e_{h(a)} \,,\\
U^-(a)&=
e_{t(a)}\otimes ae_{h(a)}=e_{t(a)}\otimes a\,.
  \end{aligned}
\end{equation} 
Hence, combining this with \eqref{Eq:dgal} we find:  
\begin{equation*}  \label{Eq:dgalfin}
  \begin{aligned}
\dgal{a,a}_{\PP}=\dgal{a,a}_{\PP_a}=&
\frac{1}{2} U^+(a)'U^-(a)''\otimes U^-(a)'U^+(a)'' 
-\frac{1}{2} U^-(a)'U^+(a)''\otimes U^+(a)'U^-(a)'' \\
&=\frac{1}{2}  a^2\otimes e_{t(a)}e_{h(a)} 
-\frac{1}{2} e_{t(a)}e_{h(a)}\otimes a^2\,.
  \end{aligned}
\end{equation*} 
Since $e_{t(a)}e_{h(a)}=e_{t(a)}=e_{h(a)}$ if $a$ is a loop and zero otherwise,  we obtain \eqref{loop} for $a\in Q$. Similar calculation for $\dgal{a^*, a^*}_{\PP}$ leads to the same result with the overall minus, which explains $\epsilon(a)$ in \eqref{loop}.  
\end{proof}

\bigskip

For any $a\in A$, define $\tr (a)=\sum_{i=1}^{|\alpha|} a_{ii}$; this is a $\Gl_\alpha$-invariant function on $\Rep_B(A, \alpha)$. We have the following useful formula \cite[Proposition 7.7.3]{VanDenBergh}:
\begin{equation}\label{trace}
\{\tr (a), \tr (b)\}=\tr \{a, b\}\,.
\end{equation}
Here the bracket on the left is the one induced on $\mathcal O (\Rep_B(A, \alpha))$ by \eqref{derr}, while $\{a,b\}$ in the right-hand side stands for the bracket on $A$ obtained from the double bracket:  
\begin{equation}\label{sbra}
\{a,b\}=m\circ \dgal{a, b}=\dgal{a,b}'\dgal{a,b}''\,.  
\end{equation}

\bigskip

We finish this subsection with two remarks. The first remark is that a total ordering on $\bar{Q}$ is not necessary to define a quasi-Hamiltonian structure on $A$. Indeed, a bivector $\PP$ can be obtained by fusion, see \cite[Proof of Theorem 6.7.1]{VanDenBergh}. The construction only requires to order arrows that start at any given vertex $i$, and the resulting bivector $\PP$ can be written as
\begin{equation}\label{Eq:PPi}
 \PP=\frac{1}{2}\sum_{a\in \bar{Q}} \epsilon(a) (1+a^*a)\frac{\partial}{\partial a} 
\frac{\partial}{\partial a^*} \,  - \, \frac12 \sum_{i\in I} \,  \sum_{\genfrac{}{}{0pt}{2}{a<b\in \bar{Q}}{t(a)=t(b)=i}} F_aF_b
\,,
\end{equation}
where 
\begin{equation*}
F_a=\frac{\partial}{\partial a^*}a^*-a
\frac{\partial}{\partial a}\,.
\end{equation*}
The moment map $\Phi$ is also well-defined, since each $\Phi_i=e_i \Phi e_i$ can be written as
\begin{equation}\label{pphii}
\Phi_i=\prod_{a\in \bar{Q}\,,\ t(a)=i} (e_i+aa^*)^{\epsilon(a)}\,,
\end{equation}
so the order of the factors is only needed to be prescribed at each vertex.

Another remark is about a slight modification of Theorem \ref{Thm:QStruct} which will be useful later. Let $S\subset \bar{Q}$ such that for all $a\in S$ we have $a^*\in S$. Write $1_S:\bar{Q}\to\{0,1\}$ for the characteristic function of the subset $S$; we have $1_S(a)=1_S(a^*)$ for all $a\in\bar{Q}$.

\begin{thm}
\label{Thm:QStructloc}
 Let $A_S$ be obtained from $\CC \bar{Q}$ by inverting all elements $(1_S(a)+aa^*)_{a\in \bar{Q}}$. The algebra $A_S$ has a quasi-Hamiltonian structure given by 
\begin{equation}\label{Eq:PPloc}
 \PP_S=\frac{1}{2}\sum_{a\in \bar{Q}} \epsilon(a) (1_S(a)+a^*a)\frac{\partial}{\partial a} 
\frac{\partial}{\partial a^*} \, - \,\frac12 \sum_{a<b\in \bar{Q}} \left(\frac{\partial}{\partial a^*}a^*-a
\frac{\partial}{\partial a}\right)\left(\frac{\partial}{\partial b^*}b^*-b\frac{\partial}{\partial b}\right) 
\,,
\end{equation}
and the multiplicative moment map $\Phi_S=(\Phi_{S,i})_{i\in I}$ given by
\begin{equation}\label{pphiloc}
\Phi_{S,i}=\prod_{a\in \bar{Q}\,,\ t(a)=i} (1_S(a)e_i+aa^*)^{\epsilon(a)}\,.
\end{equation}
\end{thm}

\begin{proof}
In the case $S=Q$ this is Theorem \ref{Thm:QStruct}, and the general case can be proved by the same method. Alternatively, $A_S$ can be obtained from $A$ by inverting arrows in ${Q}$. Indeed, take $a: i\to j$ in $Q\setminus S$ and adjoin to $A$ the formal inverse of $a$, i.e. an element ${a}^{-1}$ such that $e_j{a}^{-1}={a}^{-1}e_i={a}^{-1}$, $a{a}^{-1}=e_i$, ${a}^{-1} a=e_j$. Set $\hat{a}=a^{-1}+a^*$, then $e_i+aa^*=a\hat{a}$ and $e_j+a^*a=\hat{a}a$. Define $\hat{\hat{a}}=a$ for $a\in Q\setminus S$, while for $a\in S$ set $\hat{a}:=a^*$. Therefore, for all $a\in Q$ the elements $1_S(a)e_{t(a)}+a\hat{a}$ are invertible in $A_S$. Inside $A_S$ we have a subalgebra (over $B$), isomorphic to $\CC\bar{Q}$ and generated by all $a, \hat{a}$ with $a\in Q$. One can now define double derivations $\frac{\partial}{\partial a}$ and $\frac{\partial}{\partial \hat{a}}$ on this subalgebra, in the same way as we define $\frac{\partial}{\partial a}$, $\frac{\partial}{\partial {a^*}}$ on $\CC\bar{Q}$ in \eqref{Eq:Dba}. One can think of this as a change of variables from $a, a^*$ to $a, \hat{a}$; then a simple calcualtion shows that under this change of variables,  
\begin{equation*}
 \frac{\partial}{\partial a}\mapsto \frac{\partial}{\partial a}-a^{-1}\frac{\partial}{\partial\hat{a}}a^{-1}\,,\quad \frac{\partial}{\partial a^*}\mapsto\frac{\partial}{\partial \hat{a}}\,. 
\end{equation*}
Substituting this into \eqref{Eq:PP} and then renaming $\hat{a}$ as $a^*$, we obtain (modulo graded commutators) the bivector \eqref{Eq:PPloc}, as claimed.
\end{proof}
Explicit formulas for the double bracket on $A_S$ are almost the same, the only difference is \eqref{aast} which gets replaced by
\begin{equation}\label{aastloc}
\dgal{a,a^*}\,=\, 1_S(a) e_{h(a)}\otimes e_{t(a)}
+\frac{1}{2} a^*a\otimes e_{t(a)} 
+\frac{1}{2} e_{h(a)}\otimes aa^*
+\frac{1}{2} (a^*\otimes a-a\otimes a^*)\delta_{h(a),t(a)}\,.
\end{equation}

\subsection{Multiplicative quiver varieties}\label{Sec:mqv}
Let us turn now to representation spaces. Below we will always work in a relative setting, and from now on we will drop the subscript $B$ from the notation. For instance, given $\alpha\in\N^I$, a representation $\Rep(\CC\bar{Q}, \alpha)$ will always mean $\Rep_B(\CC\bar{Q}, \alpha)$, where $B=\oplus_i \CC e_i$, with $e_i$ acting as the identity on $V_i=\CC^{\alpha_i}$. For each arrow $a\in\bar{Q}$, we have $a=e_{t(a)}ae_{h(a)}$, therefore, $a$ is represented by a matrix with at most one non-zero block of size $\alpha_{t(a)}\times \alpha_{h(a)}$. Therefore, this can be viewed as a \emph{quiver representation}, consisting of vector spaces $\VV_i=\CC^{\alpha_i}$, $i\in I$ and linear maps $X_a: \VV_{h(a)}\to \VV_{t(a)}$ for each $a\in Q$.\footnote{If this looks to the reader as a representation of the opposite quiver, that is because of our convention for composing arrows, see Remark \ref{rem2.1}.}  With this interpretation, we have
\begin{equation}\label{xa}
X_a\in \Mat_{\alpha_{t(a)}, \alpha_{h(a)}} (\CC)\,, \qquad \Rep(\CC\bar{Q}, \alpha)\cong \prod_{a\in \bar{Q}} \Mat_{\alpha_{t(a)}, \alpha_{h(a)}} (\CC).
\end{equation}
Next, $\Rep(A, \alpha)$ is an affine open subset of $\Rep(\CC\bar{Q}, \alpha)$, so it is smooth. By Theorems \ref{Thm:Poiss}, \ref{Thm:QStruct}, this is a quasi-Hamiltonian manifold, with the quasi-Poisson bracket determined by $\PP$ and with a multiplicative moment map $\mathcal{X}(\Phi)$. The representation space $\Rep(\Lambda^q, \alpha)$ is a level set $\{\Phi=q\}$ of the momentum map, so it is a closed affine subvariety in $\Rep(A, \alpha)$. Let $q^{\alpha}=\prod_{i\in I} q^{\alpha_i}$. Then $\Rep(\Lambda^q, \alpha)$ is empty unless $q^\alpha=1$ \cite[Lemma 1.5]{CBShaw}. Isomorphism classes of representations correspond to orbits under the group 
\begin{align}\label{galpha}
G(\alpha)=\bigg(\prod_{i\in I} \Gl_{\alpha_i}\bigg)/\,{\CC}^\times\,,
\end{align}
acting by conjugation. Here ${\CC}^\times$ denotes the diagonal subgroup of scalar matrices. Semi-simple representations correspond to closed orbits.

Consider the affine variety $\mathcal{S}_{\alpha, q}:=\Rep(\Lambda^q, \alpha)//G(\alpha)$, whose points correspond to semi-simple representations of $\Lambda^{q}$ of dimension $\alpha$. We will mostly deal with the situation when $q$ and $\alpha$ are such that all representations in $\Rep(\Lambda^q, \alpha)$ are simple. In this case, we have the following result which is a combination of ~\cite[Theorems 1.8 \&1.10]{CBShaw} and \cite[Proposition 1.7]{VanDenBergh}.
\begin{thm} 
\label{dim}
Let $p(\alpha)=1+\sum_{a\in Q}\alpha_{t(a)}\alpha_{h(a)}-\alpha\cdot\alpha$, where $\alpha\cdot\alpha=\sum_{i\in I}\alpha_i^2$. Suppose that $\Rep(\Lambda^q, \alpha)$ is non-empty and all representations in $\Rep(\Lambda^q, \alpha)$ are simple. Then $\alpha$ is a positive root of $Q$ and $\Rep(\Lambda^q, \alpha)$ is a smooth affine variety of dimension $g+2p(\alpha)$, with $g=\dim G(\alpha)=\alpha\cdot\alpha-1$. The group $G(\alpha)$ acts freely on $\Rep(\Lambda^q, \alpha)$, so 
$\mathcal S_{\alpha, q}=\Rep(\Lambda^q, \alpha)/G(\alpha)$ is a Poisson manifold of dimension $2p(\alpha)$, obtained by quasi-Hamiltonian reduction.
\end{thm}

The Poisson bracket on $\mathcal O(\mathcal S_{\alpha, q})=\mathcal O(\Rep(\Lambda^q, \alpha))^{G(\alpha)}$ is obtained from \ref{Prop:BrInduced}, \ref{Thm:QStruct}. Moreover, it follows from \cite[8.3.1]{VdBquasiSpaces} that this Poisson bracket is non-degenerate, so the variety $\mathcal S_{\alpha, q}$ is, in fact, symplectic. The varieties $\mathcal S_{\alpha, q}$ are sometimes referred to as \emph{multiplicative quiver varieties}, though we will reserve this name for a special case related to framed quivers, described below.

\medskip
  
Let $Q$ be an arbitrary quiver with the vertex set $I$. A \emph{framing} of $Q$ is a quiver $\widetilde{Q}$ that has one additional vertex, denoted $\infty$, and a number of arrows $i\to \infty$ from the vertices of $Q$ (multiple arrows are allowed). Given $\alpha\in \N^I$ and $q\in({\CC^\times})^I$, we extend them from $Q$ to $\widetilde{Q}$ by putting $\alpha_\infty=1$ and $q_\infty=q^{-\alpha}$. Thus, we put
\begin{align}
\label{frame}
\aalpha=(1, \alpha)\,, \qquad \qq=q^{-\alpha}e_{\infty}+\sum_{i\in I} q_ie_i\,.
\end{align}
Consider the representation space $\Rep(\Lambda^{{\qq}}, \aalpha)$ for the multiplicative preprojective algebra of $\widetilde{Q}$. The quotients
\begin{align}
\label{qvar}
\mathcal M_{\alpha, q} (Q):=\Rep(\Lambda^{\qq}, \aalpha)//G(\aalpha)
\end{align} 
are called {multiplicative quiver varieties}. Note that since $\alpha_\infty=1$, we have 
\begin{align}
\label{gl}
G(\aalpha)=\bigg(\prod_{i\in I\sqcup\{\infty\}} \Gl_{\alpha_i} \bigg)\big/{\CC}^\times \cong \prod_{i\in I} \Gl_{\alpha_i}=\Gl_{\alpha}\,.
\end{align}

We say that $q=\sum_{i\in I} q_ie_i$ is \emph{regular} if $q^\alpha\ne 1$ for any root $\alpha$ of the quiver $Q$. We have the following result, cf. \cite[Theorem 2.8]{Nak94}, \cite[Proposition 3]{BCE}. 

\begin{prop}\label{mqvar} Choose an arbitrary framing $\widetilde{Q}$ of $Q$ and let $\aalpha$ and $\qq$ be as in~\eqref{frame}. If $q$ is regular, then every module of dimension $\aalpha$ over the multiplicative preprojective algebra $\Lambda^\qq$ is simple. Hence, the group $\Gl_\alpha$ acts freely on $\Rep(\Lambda^{{\qq}}, \aalpha)$ and the multiplicative quiver variety $\mathcal M_{\alpha, q} (Q)$ is smooth.
\end{prop}

\begin{proof}
Let $V$ be a $\Lambda^{{\qq}}$-module of dimension $\aalpha=(1, \alpha)$. If $V$ is non-simple, then it has a proper submodule $U\subset V$. The dimension vector of $U$ is either of the form $(1, \beta)$ or $(0, \beta)$, for some $\beta\in \N^I$. In the latter case, by passing to a submodule, we may assume that $U$ is simple. But then $\beta$ must be a positive root of $Q$ and $q^\beta=1$, by \cite[Lemma 1.5~\&~Theorem 1.8]{CBShaw}. Therefore, $q$ cannot be regular. In the case when $\dim U=(1, \beta)$ we consider the quotient module $V/U$ and repeat the argument.  
\end{proof}

It follows that if $q$ is regular and $\mathcal M_{\alpha, q} (Q)\ne\emptyset$, then $\aalpha=(1, \alpha)$ is a positive root of $\widetilde Q$ and $\mathcal M_{\alpha, q} (Q)$ is a smooth affine variety of dimension $2p(\aalpha)$.   

\begin{rem} The varieties $\mathcal M_{\alpha, q} (Q)$ are the same as \emph{framed} multiplicative quiver varieties studied by Yamakawa \cite{Y} (with the zero stability parameter), see also the Appendix by Nakajima and Yamakawa in \cite{BEF}.    
\end{rem}


\section{Tadpole quiver} \label{tadpole} 

In this section we describe the way to obtain the trigonometric Ruijsenaars--Schneider system by a quasi-Hamiltonian reduction. The main results in this section are not new, see e.g. \cite{R88}, \cite{FockRosly}, \cite{Oblomkov}, \cite{FehKlim12}, but we provide self-contained proofs whcih will serve as a preparation for the later sections. We should stress that we focus on algebraic and geometric aspects, working over $\CC$. Therefore, we do not address more subtle questions about various real, compact and non-compact forms of the complexified system, see \cite{R88}, \cite{FehKlim12} and references therein. Note that the choice of a real form is crucial for studying the particle dynamics and scattering, cf. \cite{R88}.   
    
Let $Q$ be a tadpole quiver with vertices $\{\infty,0\}$ and two arrows, $x:0\to 0$ and
$v:0 \to \infty$. Let us write $y=x^*$, $w=v^*$ for the opposite arrows. 
We choose an ordering $x<y<v<w$ on $\bar{Q}$. As before, form an algebra $A$  by adjoining $\{(1+aa^*)^{-1}\}_{a\in\bar{Q}}$ to the path algebra $\CC\bar{Q}$. The quasi-Poisson bracket on $A$ is given in Proposition \ref{Pr:dbr}:
 \begin{subequations}
       \begin{align}
\dgal{x,x}\,=\,&\frac{1}{2}\left( x^2\otimes e_{0}- e_{0}\otimes x^2 \right)\,,\quad
\dgal{y,y}\,=\,\frac{1}{2}\left( e_0\otimes y^2- y^2\otimes e_0 \right)\,,\label{tadida}\\
\dgal{x,y}\,=\,&e_{0}\otimes e_{0}
+\frac{1}{2} yx\otimes e_{0} +\frac{1}{2} e_{0}\otimes xy
+\frac{1}{2} (y\otimes x-x\otimes y)\,, \label{tadidb}\\
\dgal{v,v}\,=\,&\dgal{w,w}=0\,,\quad \dgal{v,w}= e_\infty\otimes e_0+ \frac12 e_\infty \otimes vw+\frac12 wv\otimes e_0\,,\label{tadidc}\\
\dgal{x, v}\,=\,& \frac12 e_{0}\otimes xv-\frac12 x\otimes v\,,\quad \dgal{x, w}= \frac12 wx\otimes e_0-\frac12 w\otimes x\,,\label{tadidd}\\
\dgal{y, v}\,=\,& \frac12 e_{0}\otimes yv-\frac12 y\otimes v\,,\quad \dgal{y, w}= \frac12 wy\otimes e_0-\frac12 w\otimes y\,.\label{tadide} 
\end{align}
  \end{subequations}
If we further localise $A$ by adding $x^{-1}$, then we can replace $y$ by $z=y+x^{-1}$, and the brackets between $x, z$ are very similar, cf. Proposition \ref{Thm:QStructloc} and \eqref{aastloc}:  
\begin{equation}
\dgal{z,z}\,=\,\frac{1}{2}\left( e_0\otimes z^2- z^2\otimes e_0 \right)\label{tadidaz}\,,\quad
\dgal{x,z}\,=\, \frac{1}{2} zx\otimes e_{0} +\frac{1}{2} e_{0}\otimes xz
+\frac{1}{2} (z\otimes x-x\otimes z)\,.
\end{equation}

Let us calculate some further brackets. Let $\{-,-\}$ denote the bracket $A\times A\to A$ defined by \eqref{sbra}. This bracket is not anti-symmetric in general, but it satisfies Leibniz's rule in the second argument, and by \cite[2.4 \& Proposition 5.1.2]{VanDenBergh} $A$ is a left Loday algebra. I.e., we have the following identities:
\begin{equation}\label{loday}
\{a, bc\}=\{a, b\}c+b\{a,c\}\,,\qquad \{a, \{b, c\}\}=\{\{a, b\}, c\}+\{b, \{a, c\}\}\,.
\end{equation}

\begin{prop}\label{Prop:int}
We have the following identities in $A$ for all $a,b\ge 0$:
\begin{equation*}
\{x^a, x^b\}=0\,,\quad \{y^a, y^b\}=0\,, \quad \{(xy)^a, (xy)^b\}=0\,. 
  \end{equation*}
If we further localise $A$ on $x$, then we also have 
\begin{equation*}
\{z^a, z^b\}=0\,,\quad z=y+x^{-1}\,. 
  \end{equation*}
\end{prop}

\begin{prop}
\label{Prop:dBrTad} Let $A'$ denote the algebra $A$ localised on $x$, and $z=y+x^{-1}$. 
For any $a, b\ge 0$ we have
  \begin{subequations}
       \begin{align}
\{x^a, zx^b\}&=azx^{a+b} \mod{[A',A']}\,,\label{Eq:tadfg}\\
\{zx^a, zx^b\}&=\sum \limits_{r=1}^{a} zx^rzx^{a+b-r}- \sum \limits_{r=1}^{b} zx^rzx^{a+b-r} \mod{[A',A']}\,.\label{Eq:tadgg}
   \end{align}
  \end{subequations}
\end{prop} 
Proofs can be found in Appendix \ref{Ann:brackets}. \qed

\bigskip

For $q=(q_\infty, q_0)$, the multiplicative 
preprojective algebra $\Lambda^{q}$ is the quotient of $A$  by the relation 
\begin{equation}
\label{Eq:ConTad}
 (1+xy)(1+yx)^{-1}(1+vw)(1+wv)^{-1}=q_0 e_0 + q_\infty e_\infty\,.
\end{equation}
Multiplication by the idempotents $e_\infty, e_0$ turns this into two relations:
\begin{subequations}
       \begin{align}
&(e_0+xy)(e_0+yx)^{-1}(e_0+vw)=q_0e_0\,,\label{rel1}\\
&(e_\infty+wv)^{-1}=q_\infty e_\infty\,.\label{rel2}
       \end{align}
  \end{subequations} 
Choose a dimension vector $\alpha=(1,n)$ and set $q_\infty=q_0^{-n}$ to satisfy $q^\alpha=1$. A representation of $\Lambda^q$ of dimension $\alpha$ is a pair $(\VV_\infty, \VV_0)=(\CC, \CC^n)$ together with linear maps representing arrows of $\bar{Q}$ and satysfying \eqref{rel1}, \eqref{rel2}. 
Denote the matrices representing the arrows as $X, Y, V, W$. Therefore, points of $\Rep(\Lambda^q, \alpha)$ are represented by quadruples $(X,Y,V,W)$,
\begin{equation*}
  X,Y\in \Mat_{n\times n}(\CC),\quad V\in \Mat_{n\times 1}(\CC),\quad W\in \Mat_{1\times n}(\CC)\,,
\end{equation*}
satisfying 
  \begin{subequations}
       \begin{align}
&(\Id_{n}+XY)(\Id_{n}+YX)^{-1}(\Id_{n}+VW)=q_0\Id_n\,, \label{Eq:CondTadInv} \\
&(1+WV)^{-1}=q_\infty\quad(q_\infty=q_0^{-n})\,. \label{Eq:CondRq}
       \end{align}
  \end{subequations} 
The group $\Gl_n(\CC)$ acts on these quadruples by
\begin{equation}\label{gact}
g. (X,Y,V,W)=(gXg^{-1},gYg^{-1}, gV, Wg^{-1})\,,\quad g\in\Gl_n\,, 
\end{equation}
and the orbits in $\Rep\left(\Lambda^{q},\alpha\right)/\!/\Gl_n$ correspond to isomorphism classes of semisimple representations. 
Introduce the \emph{Calogero--Moser space} $\mathcal{C}_{n,q_0}$ as 
\begin{equation*}
\mathcal{C}_{n,q_0}=\Rep\left(\Lambda^{q},\alpha\right)/\!/\Gl_n\,.
\end{equation*}
This is a multiplicative quiver variety for a framed one-loop quiver, and applying the results of \ref{Sec:mqv}, we have

\begin{prop}
Suppose $q_0$ is not a root of unity. Then the group $\Gl_n$ acts on $\Rep(\Lambda^q, \alpha)$ freely, and $\mathcal{C}_{n,q_0}$ is a smooth symplectic variety of dimension $2n$.
\end{prop}

The variety $\mathcal{C}_{n,q_0}$ admits a description in terms of pairs of matrices as follows:
\begin{equation*}
{\mathcal{C}}_{n,q_0}=\{X,Y\in\Mat_{n\times n}(\CC) \mid \mathrm{rank}\left((\Id_n+XY)(\Id_n+YX)^{-1}-q_0\Id_n\right)=1\}//\Gl_n\,. 
\end{equation*}
We may also consider the open subset $\mathcal C^{0}_{n, q_0}\subset \mathcal C_{n, q_0}$ on which $X$ is invertible. Introducing $Z:=Y+X^{-1}$, we have $\Id_n+XY=XZ$, $\Id_n+YX=ZX$ and therefore
\begin{equation*}
{\mathcal{C}}^0_{n,q_0}=\{X,Z\in\Gl_{n} \mid \mathrm{rank} \left(XZX^{-1}Z^{-1}-q_0\Id_n\right)=1\}//\Gl_n\,. 
\end{equation*}

The Poisson bracket on $\mathcal O(\mathcal C_{n, q_0})=\mathcal O(\Rep(\Lambda^q, \alpha))^{\Gl_n(\CC)}$ is induced by the double bracket on $A$. Proposition \ref{Prop:int} together with \eqref{trace} give us

\begin{thm}
\label{Thm:TadInvol}
The following families of functions on $\mathcal{C}_{n,q_0}$ are Poisson commuting:  
\begin{equation*}
\left\{\tr X^j\,\, \big|\,\, j\in \N\right\}\,,\quad \left\{\tr Y^j\big|\,\, j\in \N\right\}\,,\quad \left\{\tr (1+XY)^j\big|\,\, j\in \Z\right\}\,,\quad \left\{\tr (Y+X^{-1})^j\big|\,\, j\in \Z\right\}\,,
\end{equation*}
where the last family is viewed on $\mathcal C^{0}_{n, q_0}\subset \mathcal C_{n, q_0}$.
 \end{thm}
\begin{rem}
If we assume $Y$ invertible, we can get another commuting family $\left\{\tr (X+Y^{-1})^j\big|\,\, j\in \Z\right\}$.   
\end{rem}
To interpret these families as integrable particle systems, we next write them down in suitable canonical (Darboux) coordinates. 

\subsection{Darboux coordinates}   \label{ss:TadCoord} We take $\h=\CC^n$ with coordinates $x_1, \dots, x_n$, and define
\begin{equation*}
\hreg=\{x\in\h\, \mid \, x_i\ne 0\,,\ x_i\ne x_j\,,\ x_i\ne q_0x_j\ \text{for all}\ i\ne j\}\,.  
\end{equation*}
Let $\h^\times=(\CC^\times)^n$ with coordinates $\nu_1,\dots, \nu_n\in\CC^\times$. We are going to define a map 
\begin{equation*}
\xi:\ \hreg\times\h^\times\to \mathcal C^0_{n, q_0}\,.
\end{equation*} 
Given $x\in\hreg$, $\nu\in\h^\times$, we set $\xi(x, \nu)=(X,Z)$ where  
\begin{equation*}
X=\diag(x_1,\ldots,x_n)\,,\quad Z=(Z_{ij})\,,\ \text{with }\  Z_{ij}=\frac{(q_0-1)\nu_j}{q_0-x_i/x_j}\,.
\end{equation*}
One can check directly, using Cauchy's determinant formula, that $Z$ is invertible. We also have
\begin{equation*}
(XZ-q_0ZX)_{ij}=x_iZ_{ij}-q_0Z_{ij}x_j=(1-q_0)\nu_jx_j\,,
\end{equation*} 
which shows that the matrix $XZX^{-1}Z^{-1}-q_0\Id_n$ has rank one, and so the pair $(X,Z)$ determines a point of $\mathcal{C}^0_{n,q_0}$.
Moreover, if one simultaneoulsy permutes $x_i$ and $\nu_i$, the resulting $(X,Z)$ get conjugated by the matrix of that permutation. Therefore, we have in fact a map
\begin{equation*}
\xi:\ \hreg\times\h^\times\,/\,S_n\to \mathcal C^0_{n, q_0}\,.
\end{equation*} 
It is easy to see that $\xi$ is injective, and since $\dim\mathcal C_{n, q_0}=2n$, we can use $(x_1,\dots, x_n, \nu_1, \dots, \nu_n)$ as local coordinates on $\mathcal C^0_{n, q_0}$. Note that by \cite{Oblomkov}, the variety $\mathcal C_{n, q_0}^0$ is connected, hence the image of $\xi$ is dense.   

\begin{prop}[\cite{FockRosly}, \cite{Oblomkov}]
\label{Thm:IsoPssTad}
The local diffeomorphism $\xi:\, \hreg\times\h^\times \,/ \,S_n\to \mathcal{C}^0_{n,q_0}$ becomes a Poisson map if we equip $\hreg\times \h^\times$ with 
the following $S_n$-invariant Poisson bracket $\{-,-\}'$: 
  \begin{subequations}
       \begin{align}
 \{x_i,x_j\}'&=0 \, , \label{Eq:Txx} \\
\{x_i,\nu_j\}'&=\delta_{ij} x_i \nu_j\, , \label{Eq:Txnu} \\
 \{\nu_i,\nu_j\}'&=\frac{(1-q_0)^2 (x_i+x_j)x_ix_j\nu_i \nu_j}{(x_i-x_j)(x_i-q_0x_j)(x_j-q_0x_i)} 
\label{Eq:Tnunu}\,. 
       \end{align}
  \end{subequations}
\end{prop}
\begin{proof}
We need to check that the map $\xi$ satisfies 
$\xi^*\{f, g\}=\{\xi^*f,\xi^*g\}'$ for any $f,g\in \mathcal O(\mathcal C^0_{n, q_0})$. Since the functions $f_a:=\tr(X^a)$ and $g_b:=\tr (ZX^b)$ with $a,b=1,\dots, n$ 
form local coordinates at a generic point, it suffices to check the Poisson property for these functions (cf. the proof of Proposition 2.7 in \cite{EtingCM}). From \eqref{trace} and Propositions \ref{Prop:int}, \ref{Prop:dBrTad} we have:
\begin{equation}\label{old}
\{f_a, f_b\}=0\,,\quad \{f_a, g_b\}=ag_{a+b}\,,\quad \{g_b, g_c\}=\sum_{r=b+1}^c h_{r, b+c-r}\,,
\end{equation}
where $h_{r,s}:=\tr (ZX^rZX^s)$ (we assume $b<c$ in the last formula). Next,
\begin{equation}\label{hab}
\xi^*f_a=\sum_{i} x_i^a\,,\quad \xi^*g_b=\sum_{i} \nu_ix_i^b\,,\quad \xi^*h_{r, a+b-r}=\sum_{i,j}\frac{(q_0-1)^2\nu_i\nu_j}{(q_0-x_i/x_j)(q_0-x_j/x_i)}x_j^rx_i^{a+b-r}\,.
\end{equation}
Therefore,
\begin{equation}\label{sumh}
\sum_{r=b+1}^a \xi^*h_{r, a+b-r}=(a-b)\sum_i \nu_i^2x_i^{a+b} +\sum_{i\ne j}\frac{(q_0-1)^2\nu_i\nu_j}{(q_0-x_i/x_j)(q_0-x_j/x_i)}\sum_{r=b+1}^a x_j^rx_i^{a+b-r}\,.
\end{equation}
Now notice that \eqref{Eq:Txx} and \eqref{Eq:Txnu} give 
\begin{equation*} 
    \{\xi^*f_a,\xi^*f_b\}'=\sum_{i,j} \{x_i^a,x_j^b\}'=0 \, ,\qquad
\{\xi^*f_a,\xi^*g_b\}'=\sum_{i,j} \{x_i^a,\nu_jx_j^b\}'=a \sum_{i} x_i^{a+b}\nu_i\,,
\end{equation*}
which agrees with the first two formulas in \eqref{old}. 
Next, we use \eqref{Eq:Tnunu} to find that   
\begin{equation*} 
\{\xi^*g_a,\xi^*g_b\}'=\sum_{i,j} \{\nu_ix_i^a,\nu_jx_j^b\}' =\sum_{i\neq j} \frac{(1-q_0)^2 (x_i+x_j)\nu_i \nu_j x_i^ax_j^b}{(x_i-x_j)(x_i/x_j-q_0)(x_j/x_i-q_0)} 
+(a-b) \sum_{i}  \nu_i^2x_i^{a+b}\,. 
\end{equation*}
It is easy to see that this expression coincides with \eqref{sumh}. 
\end{proof}

\bigskip

The coordinates $(x_i, \nu_i)$ are not yet canonical since $\{\nu_i, \nu_j\}'\ne 0$. A set of log-canonical coordinates can be constructed analogously to \cite{FockRosly}. Namely, introduce
\begin{equation} 
\label{Eq:TadSigm_i}
 \sigma_i=\nu_i\,\prod_{j:\, j\neq i}\frac{1-x_ix_j^{-1}}{1-q_0x_ix_j^{-1}}\qquad (i=1, \dots, n)\,. 
\end{equation}
Then one checks directly that
\begin{equation*}
\{x_i,x_j\}'=0 \, , \quad
\{x_i,\sigma_j\}'=\delta_{ij} x_i \sigma_j\, , \quad
 \{\sigma_i,\sigma_j\}'=0\,. 
 \end{equation*}

In these coordinates, we can identify some of the known integrable particle systems among those in Theorem \ref{Thm:TadInvol}. The best known is the one corresponding to the Hamiltonians $\tr (Y+X^{-1})^j$. Namely, after writing $Y+X^{-1}$ in coordinates $x_i, \sigma_i$, we have
\begin{equation}\label{RSL}
Y+X^{-1}=Z\,,\quad \text{where}\ Z_{ij}=\, \sigma_j \,\frac{q_0-1}{q_0-x_ix_j^{-1}}\,\prod_{k\neq j} \,\frac{1-q_0x_jx_k^{-1}}{1-x_jx_k^{-1}}\,.
\end{equation} 
This gives
\begin{equation*}
\tr Z=\sum_{i=1}^n \sigma_i \prod_{j\neq i} \,\frac{1-q_0x_ix_j^{-1}}{1-x_ix_j^{-1}}\,,
\end{equation*}
which is equivalent to the classical trigonometric Ruijsenaars--Schneider Hamiltonian. In fact, $Z$ after conjugation by $\diag(\sqrt{\nu_1},\ldots,\sqrt{\nu_n})$ and a change of notation coincides with the Lax matrix from \cite[Section 4]{Ruijs86}. 

Other families in Theorem \ref{Thm:TadInvol} lead to closely related integrable systems. The family $\{\tr X^j\}_{j\in\N}$ is trivial, while the simplest Hamiltonians for the other two are:
\begin{equation*}
\tr Y=\sum_{i=1}^n \sigma_i \prod_{j\neq i} \,\frac{1-q_0x_ix_j^{-1}}{1-x_ix_j^{-1}}-\sum_{i=1}^n x_i^{-1}\,,\quad \tr(1+XY)=\sum_{i=1}^n \sigma_i x_i\prod_{j\neq i} \,\frac{1-q_0x_ix_j^{-1}}{1-x_ix_j^{-1}}\,.
\end{equation*}
The second Hamitonian $\tr (1+XY)$ reduces to the Ruijsenaars--Schneider Hamiltonian by a canonical change of variables $\tilde x_i=x_i$, $\tilde\sigma_i=\sigma_ix_i$. The first Hamiltonian $\tr Y$, as explained in \cite[Remark 3.25]{BEF}, is related to the quantum system introduced by J.F. van Diejen \cite{vD, vDE} and a closely related system considered by Baker and Forrester \cite{BF}. In slightly different canonical coordinates, the classical Hamiltonian system described by $\tr Y$ was also considered by P. Iliev \cite{Iliev00}. See also recent works \cite{M, FG, FM} devoted to the study of some special cases of the van Diejen's system.

Note that the interpretation of these systems through quasi-Hamiltonian reduction achieves two things at once: a completed phase space allowing
particles to coalesce, and explicit dynamics on the affine space $\Rep(A, \alpha)$. Since this is mostly well-familiar (cf. \cite{Ruijs86, FockRosly, CF, Iliev00}), we skip the details (see Proposition \ref{dynam} below for the more general systems related to the cyclic quivers). Another benefit is that we can see rather easily the \emph{self-duality} for this system, originally established by Ruijsenaars.

\begin{prop}[\cite{R88}, \cite{FehKlim12}]\label{duality}
The transformation $(X,Z)\mapsto (Z,X)$ induces a Poisson map $\varphi: \mathcal{C}_{n, q_0}^0\,\to\,\mathcal{C}_{n, q_0^{-1}}^0$ (more precisely, $\varphi$ changes the sign of the Poisson bracket).
\end{prop}
\begin{proof} 
From the proof of Proposition \ref{Thm:IsoPssTad}, we know that the Poisson bracket on $\mathcal{C}_{n, q_0}^0$ is completely determined by the brackets between the functions $\tr\,(X^aZ^b)$ with $a,b\ge 0$. These brackets, on the other hand, are determined by the antisymmetric biderivation (bivector) on $\Rep(\CC\bar Q, \alpha)$, defined according to \eqref{derr}. 
Using this and \eqref{tadida}, \eqref{tadidaz}, we find that
 \begin{eqnarray*}
\{X_{ij},X_{uv}\}\,&=\,&\frac{1}{2}(X^2)_{uj}\delta_{iv}- \frac12 \delta_{uj}(X^2)_{iv}\,,\quad
\{Z_{ij},Z_{uv}\}\,=\,\frac{1}{2} \delta_{uj}(Z^2)_{iv}- \frac12 (Z^2)_{uj}\delta_{iv}\,,\\
\{X_{ij},Z_{uv}\}\,&=\,& \frac{1}{2} (ZX)_{uj}\delta_{iv} +\frac{1}{2} \delta_{uj}(XZ)_{iv}
+\frac{1}{2} Z_{uj}X_{iv}-\frac12 X_{uj}Z_{iv}\,. 
  \end{eqnarray*} 
It is now obvious that swapping $X$ with $Z$ in these formulas leads to a change of sign. Therefore, $\{\tr X^aZ^b, \tr X^cZ^d\}=-\{\tr Z^aX^b, \tr Z^cX^d\}$ for all $a,b,c,d\ge 0$, as needed.  
\end{proof}  

For the later use, let us formulate the above proposition in coordinates. We already have log-canonical coordinates $x=(x_1,\dots, x_n)$, $\sigma=(\sigma_1, \dots, \sigma_n)$ on $\mathcal{C}_{n, q_0}^0$, obtained by taking $X=\diag(x_1, \dots, x_n)$ and $Z$ as in \eqref{RSL}. The second set of coordinates $z=(z_1,\dots, z_n)$, $\theta=(\theta_1, \dots, \theta_n)$ on the same space is given by
 \begin{equation}\label{RSL2}
Z=\diag(z_1, \dots, z_n)\,,\quad  X_{ij}=\, \theta_j\,\frac{q_0^{-1}-1}{q_0^{-1}-z_iz_j^{-1}}\,\prod_{k\neq j} \,\frac{1-q_0^{-1}z_jz_k^{-1}}{1-z_jz_k^{-1}}\,.
\end{equation}  
Then the Proposition \ref{duality} can be reformulated by saying that the transformation $(x,\sigma)\to (z, \theta)$ is anti-symplectic, i.e.
\begin{equation}\label{dual}
\{z_i,z_j\}=0 \, , \quad
\{z_i,\theta_j\}=-\delta_{ij} z_i \theta_j\, , \quad
 \{\theta_i,\theta_j\}=0\,. 
 \end{equation}

\begin{rem} From the construction, it is obvious that the transformation $(x,\sigma)\to (z, \theta)$ is an involution. 
In \cite{R88} the canonicity of that involution was established by a rather roundabout method. For a more natural geometric approach, see \cite{FehKlim12, FGNR}. 
This can also be deduced from the results of \cite{Oblomkov}, where the space $\mathcal{C}_{n, \tau}^0$ was related to the center $\mathcal Z(H_{1,\tau})$ of the double affine Hecke algebra $H_{q,\tau}$. Indeed, we have the Cherednik--Fourier transform $\varepsilon: H_{q, \tau}\to H_{q^{-1}, \tau^{-1}}$ as in \cite[3.5]{Oblomkov}. It is an algebra homomorphism, so in the classical limit $q\to 1$ it gives a Poisson map $\mathcal Z(H_{1,\tau})\to \mathcal Z(H_{1,\tau^{-1}})$ which interchanges $X$ and $Z$ (the bracket changes sign because $q$ goes to $q^{-1}$ under $\epsilon$). 
\end{rem}

\begin{rem}
The Calogero--Moser spaces $\mathcal C_{n,q_0}$ and $\mathcal C_{n,q_0}^0$ are $q$-analoqs of Wilson's Calogero--Moser spaces \cite{W}. In \cite{Iliev00} and \cite{CN} they appeared in the context of the bispectral problem. They have also been studied from the point of view of non-commutative geometry, as moduli spaces of non-commutative instantons \cite{KKO} and ideals of the quantum torus algebra (see \cite{BRT} and references therein).
\end{rem}

\section{Cyclic quivers} \label{cyclic}  

For $m\ge 2$, let $Q$ be a framed cyclic quiver with the arrows $x_i:\, i\to i+1$, $i\in I=\Z/m\Z$, and $v:\, 0\to \infty$. Write $y_i=x_i^*:\, i+1\to i$ and $w=v^*:\,\infty\to 0$ for the opposite arrows. We choose the following ordering of arrows around each vertex:
\begin{equation*}
x_{i-1}<y_{i-1}<x_i<y_i\quad\text{for $i\ne 0$}\,,\qquad x_{m-1}<y_{m-1}<x_0<y_0<v<w\ \text{at $i=0$}\,.
\end{equation*}
As before, we form an algebra $A$ by adjoining to $\CC\bar{Q}$ the elements $(1+aa^*)^{-1}$ for all $a\in\bar{Q}$. Let $\dgal{-,-,}$ be a double bracket on $A$, associated to the bivector \eqref{Eq:PPi}. It gives the following brackets between the arrows $x_i, y_j$:
 \begin{subequations}
       \begin{align}
\dgal{x_i, x_j}\,=\,&\frac{1}{2} e_{i+1}\otimes x_ix_{i+1}\,\delta_{j, i+1}-\frac12 x_{i-1}x_i\otimes e_i\,\delta_{j, i-1}\,, \\
\dgal{y_i, y_j}\,=\,&\frac12 y_{i+1}y_i\otimes e_{i+1}\,\delta_{j, i+1}-\frac{1}{2} e_{i}\otimes y_iy_{i-1}\,\delta_{j, i-1}\,, \\
\dgal{x_i,y_j}\,=\,&\delta_{i,j}\left(e_{i+1}\otimes e_{i}+\frac 12 y_ix_i\otimes e_i+\frac 12e_{i+1}\otimes x_iy_i\right)\nonumber\\
\,&-\frac{1}{2} y_{i+1}\otimes x_i\,\delta_{j, i+1} + \frac{1}{2} x_i\otimes y_{i-1}\,\delta_{j, i-1}\,,\\
\dgal{y_i,x_j}\,=\,&-\delta_{i,j}\left(e_{i}\otimes e_{i+1}+\frac 12 x_iy_i\otimes e_{i+1}+\frac 12e_{i}\otimes y_ix_i\right)\nonumber\\
\,&-\frac{1}{2} y_{i}\otimes x_{i+1}\,\delta_{j, i+1} + \frac{1}{2} x_{i-1}\otimes y_{i}\,\delta_{j, i-1}\,.
\end{align}
  \end{subequations}

\medskip

Set $x=x_0+\dots +x_{m-1}$, $y=y_0+\dots +y_{m-1}$, so then $x_i=e_ix=xe_{i+1}$, $y_i=e_{i+1}y=ye_i$. If we formally invert all $x_i$, then in the localised algebra we have $x^{-1}x=xx^{-1}=1$ with $x^{-1}=\sum_{i\in I} x_i^{-1}$. Introduce the following elements $E_{\pm 1}\in A\otimes A$:
\begin{equation*}
E_1=\sum_{i\in I}e_{i+1}\otimes e_i\,,\quad E_{-1}=\sum_{i\in I}e_{i-1}\otimes e_i \,.
\end{equation*}
With this notation, we have the following easily verified formulas:
\begin{lem}
\label{cxy} 
\begin{subequations}
       \begin{align}
\dgal{x,x}\,=\,&\frac{1}{2}\left( E_1x^2-x^2 E_1\right)\,,\quad
\dgal{y,y}\,=\,\frac{1}{2}\left( y^2 E_{-1}- E_{-1}y^2\right)\,,\label{cycida}\\
\dgal{x,y}\,=\,& E_{1}
+\frac{1}{2} yx E_{1} +\frac{1}{2} E_{1}xy
-\frac{1}{2} yE_{1}x+\frac12 xE_{1} y\,,\label{cycidb}\\
\dgal{y,x}\,=\,& -E_{-1}
-\frac{1}{2} E_{-1}yx -\frac{1}{2} xyE_{-1}
+\frac{1}{2} xE_{-1} y-\frac12 yE_{-1} x\,.\label{cycidc}
\end{align}
  \end{subequations}
\end{lem} 
\begin{rem}
In the above formulas we use the outer bimodule structure on $A\otimes A$. For example, $xE_1=\sum_{i,j}x_j(e_{i+1}\otimes e_i)=\sum_{i} x_i\otimes e_{i}$.
\end{rem}
Let $\{-,-\}$ denote the bracket $A\times A\to A$ defined by \eqref{sbra}. 
\begin{prop}\label{Prop:intcyclic}
We have the following identities in $A$ for all integers $a,b\ge 0$:
\begin{equation*}
\{x^a, x^b\}=0\,,\quad \{y^a, y^b\}=0\,, \quad \{(xy)^a, (xy)^b\}=0\,. 
  \end{equation*}
If we further localise $A$ by inverting $x$, then we also have 
\begin{equation*}
\{z^a, z^b\}=0\,,\quad z=y+x^{-1}\,. 
  \end{equation*}
\end{prop}

\begin{prop}
\label{Prop:dBrcyclic}
For any $a, b, c\ge 0$ with $a, b-1, c-1 \equiv 0 \mod m$, we have
  \begin{subequations}
       \begin{align}
\{x^a, yx^b\}&=ax^{a+b-1}+ayx^{a+b} \mod{[A, A]}\,,\label{Eq:cycfg}\\
\{yx^b, yx^c\}&=(b-c) yx^{b+c-1}+ \sum \limits_{t=1}^{b} yx^tyx^{b+c-t}- \sum \limits_{t=1}^{c} yx^tyx^{b+c-t} \mod{[A, A]}\,.\label{Eq:cycdgg}
   \end{align}
  \end{subequations}
\end{prop} 
Proofs can be found in Appendix \ref{Ann:brackets}. 
 
\bigskip

For $\widetilde q=q_\infty e_\infty +\sum_{i\in I} q_ie_i$, the multiplicative 
preprojective algebra $\Lambda^{\widetilde q}$ is the quotient of $A$ by the relations: 
\begin{subequations}
       \begin{align}
&(e_i+y_{i-1}x_{i-1})^{-1}(e_i+x_iy_i)=q_ie_i\quad (i\ne 0)\,,\label{crel1}\\
&(e_0+y_{m-1}x_{m-1})^{-1}(e_0+x_0y_0)(e_0+vw)=q_0e_0\,,\label{crel2}\\
&(e_\infty+wv)^{-1}=q_\infty e_\infty\,.\label{crel3}
       \end{align}
  \end{subequations} 
 
Choose a dimension vector $\widetilde{\alpha}=(1, \alpha)$ where $\alpha\in (\N^\times)^{I}$ and set $q_\infty=q^{-\alpha}:=\prod_{i\in I}q_i^{-\alpha_i}$. Recall that $q\in(\CC^\times)^I$ is regular if $q^{\beta}\ne 1$ for any root of the (unframed) cyclic quiver. The roots for the cyclic quiver form the affine root system of type $\widetilde A_{m-1}$, therefore the regularity is equivalent to the following conditions:
\begin{equation*}
\prod_{i\le k \le j-1} q_k\ne t^n \ \text{for any $n\in\Z$ and $1\le i\le j\le m$, where}\quad t:=\prod_{i\in I}q_i\,.
\end{equation*}
We use the convention that the product in the left-hand side is empty when $i=j$, so in particular $t$ must not be a root of unity. Applying the results of \ref{Sec:mqv}, we have
\begin{prop}
For regular $q$, the variety $\mathcal M_{\alpha, q}=\Rep(\Lambda^{\qq}, \aalpha)//G(\aalpha)$, if non-empty, is smooth symplectic, of dimension $2p(\aalpha)=2\alpha_0 +2 \sum_{i\in I}\alpha_i(\alpha_{i+1}-\alpha_{i})$.
\end{prop}

A representation of $\Lambda^{\qq}$ of dimension $\alpha$ is a collection of vector spaces $\VV_\infty=\CC$ and $\VV_i=\CC^{\alpha_i}$, together with linear maps representing the arrows and satysfying the relations of \eqref{crel1} - \eqref{crel3}. Denote the matrices representing the arrows as $X_i, Y_i, V, W$, and set $X=X_0+\dots +X_{m-1}$, $Y=Y_0+\dots+ Y_{m-1}$. We can view $X,Y$ as linear endomorphisms of $\VV:=\oplus_{i\in I} \VV_i$. Proposition \ref{Prop:intcyclic} together with \eqref{trace} gives us
\begin{thm}
\label{Thm:cycInvol}
The following families of functions on $\mathcal M_{\alpha, q} $ are Poisson commuting:  
\begin{equation*}
\left\{\tr X^{jm}\,\, \big|\,\, j\in \N\right\}\,,\quad \left\{\tr Y^{jm}\big|\,\, j\in \N\right\}\,,\quad \left\{\tr (1+XY)^j\big|\,\, j\in \Z\right\}\,,\quad \left\{\tr (Y+X^{-1})^{jm}\big|\,\, j\in \Z\right\}\,,
\end{equation*}
where the last family is viewed on the open subset $\mathcal M_{\alpha, q}^{0}\subset\mathcal M_{\alpha, q}$ on which $X$ is invertible.
 \end{thm}

We can integrate explicitly the associated Hamiltonian flows in the most interesting cases:

\begin{prop}\label{dynam}
Let $t$ denote the time flow associated to $H_k:=\frac1k\tr Y^k$, $k\in m\N$. Given an initial position $X(0), Y(0), V(0), W(0)$ on $\mathcal M_{\alpha, q}$, the solution $X,Y,V,W$ at time $t$ is given by   
\begin{equation*}
X(t)=e^{-tY^k}X(0)+Y^{-1}(e^{-tY^k}-1)\,,\quad Y(t)=Y(0)\,,\quad V(t)=V(0)\,,\quad W(t)=W(0)\,.
\end{equation*}
Similarly, if $t$ denotes the time flow associated to $G_k:=\frac1k \tr Z^k$ with $Z=Y+X^{-1}$, then the solution at time $t$ is given by
\begin{equation*}
X(t)=e^{-tZ^k}X(0)\,,\quad Z(t)=Z(0)\,,\quad V(t)=V(0)\,,\quad W(t)=W(0)\,.
\end{equation*}
In both cases the flows are complete: in the first case on the whole of $\mathcal M_{\alpha, q}$, and in the second case on the open subset $\mathcal M_{\alpha, q}^{0}\subset \mathcal M_{\alpha, q}$ where $X$ is invertible.
\end{prop}
The proof of the proposition is given in Appendix \ref{Ann:dynamics}. 
\bigskip

We will be particularly interested in the special choice of the dimension vector $\alpha_i=n$ for all $i\in I$, because only in this case $\mathcal M_{\alpha, q}^{0}$ is nonempty. Accordingly, we choose $q\in (\CC^\times)^I$ and set $q_\infty=t^{-n}$, where $t=\prod_{i\in I}q_i$. With this choice, points of $\Rep(\Lambda^\qq, \aalpha)$ are represented by a collection of $X_i,Y_i, V,W$,
\begin{equation*}
  X_i, Y_i\in \Mat_{n\times n}(\CC),\quad V\in \Mat_{n\times 1}(\CC),\quad W\in \Mat_{1\times n}(\CC)\,,
\end{equation*}
satisfying 
\begin{subequations}
       \begin{align}
&(\Id_n+Y_{i-1}X_{i-1})^{-1}(\Id_n+X_iY_i)=q_i\Id_n\quad (i\ne 0)\,,\label{cmrel1}\\
&(\Id_n+Y_{m-1}X_{m-1})^{-1}(\Id_n+X_0Y_0)(\Id_n+VW)=q_0\Id_n\,,\label{cmrel2}\\
&(1+WV)^{-1}=t^{-n}\,.\label{cmrel3}
       \end{align}
  \end{subequations} 
Here all factors are assumed invertible, so the last relation can be obtained from the others after taking determinants. The group $G:=\Gl_n^m$ acts on these linear data by
\begin{equation*}
X_i\mapsto g_iX_ig_{i+1}^{-1}\,,\quad Y_i\mapsto g_{i+1}Y_ig_{i}^{-1}\,,\quad  V\mapsto g_0 V\,,\quad W\mapsto Wg_0^{-1}\,,
\end{equation*}
for any $g=(g_0, \dots, g_{m-1})\in G$. 
The space of equivalence classes of such linear data will be referred to as the \emph{Calogero--Moser space} $\mathcal{C}_{n,q}(m)$, or simply $\mathcal{C}_{n,q}$ when it does not lead to a confusion.
This is a special case of a multiplicative quiver variety, so we can apply Theorem \ref{dim} and Proposition \ref{mqvar}.
\begin{prop}
For regular $q$, the Calogero--Moser space $\mathcal{C}_{n,q}(m)$ is a smooth symplectic variety of dimension $2n$.
\end{prop}
In the next subsection we will introduce local coordinates on this space and calculate the Poisson bracket explicitly.

\subsection{Coordinates and Poisson bracket} \label{ss:CycCoord} 
We choose a regular set of parameters $q\in (\CC^\times)^I$; it will be convenient to introduce a separate notation for the quantities   
\begin{equation*}
t_s:=\prod_{0\le i\le s} q_i\quad (s=0,\dots, m-1)\,.   
\end{equation*}
We will keep using the symbol $t$ for $t_{m-1}=q_0\dots q_{m-1}$. Let us define a family of representations of $\Lambda^{\widetilde q}$ with $\widetilde q=(t^{-n}, q_0,\dots, q_{m-1})$ and of dimension $\aalpha=(1, n, \dots, n)$. We will assume that $X_i$ are invertible, so we can use $Z_i=Y_i+X_i^{-1}$ to rewrite the relations \eqref{cmrel1}--\eqref{cmrel3} as
\begin{subequations}
       \begin{align}
&(Z_{i-1}X_{i-1})^{-1}X_iZ_i=q_i\Id_n\quad (i\ne 0)\,,\label{cmrel1z}\\
&(Z_{m-1}X_{m-1})^{-1}X_0Z_0(\Id_n+VW)=q_0\Id_n\,,\label{cmrel2z}\\
&(1+WV)^{-1}=t^{-n}\,.\label{cmrel3z}
       \end{align}
  \end{subequations} 
Then by changing bases we can achieve that $X_s=\Id_n$ for $s=0,\dots, m-2$. If we introduce $A:=X_{m-1}$ and $B:=q_0^{-1}Z_0$, then from \eqref{cmrel1z} we find that
$Z_s=t_sB$ for $s=0,\dots, m-2$ and $Z_{m-1}=tA^{-1}B$. Using this in \eqref{cmrel2z} gives that $A^{-1}B^{-1}AB(1+VW)=t\Id_n$, which we can rewrite as
\begin{equation}\label{abcm}
ABA^{-1}B^{-1}(1+\widetilde V\widetilde W)=t\Id_n\,,
\end{equation}
where $\widetilde V=BAV$ and $\widetilde W=WB^{-1}A^{-1}$.
It is easy to see that isomorphic representation of $\Lambda^{\widetilde q}$ produce isomorphic quadruples $(A, B, \widetilde V, \widetilde W)$, up to the equivalence \eqref{gact}. Therefore, we obtain
\begin{prop}[cf. \cite{BEF}, Section 5.2] \label{isocm} 
Let $\mathcal C_{n,q}^0(m)$ and $\mathcal C_{n,t}^0$ be the Calogero--Moser spaces of isomorphism classes of linear data \eqref{cmrel1z}--\eqref{cmrel3z} and \eqref{abcm}, respectively. We assume that the parameters are regular, so both spaces are smooth varieties. Then the map $\xi$ sending $(A,B, \widetilde V, \widetilde W)$ to $X_s=\Id_n$, $Z_s=t_sB$ for $s=0,\dots, m-2$ and $X_{m-1}=A$, $Z_{m-1}=tA^{-1}B$, $V=A^{-1}B^{-1}\widetilde V$, $W=\widetilde WAB$ defines an isomorphism of these varieties. In particular, $\mathcal C_{n,q}^0(m)$ is connected because so is $\mathcal C_{n,t}^0$.   
\end{prop}
In fact, a stronger claim is true.  Recall, that both spaces are Poisson varieties. 
 
\begin{prop}\label{IsoPssCyc}
\label{Thm:IsoPssCyc}
The isomorphism $\xi:\, \mathcal C_{n,q}^0 \to \mathcal{C}_{n,q}^0(m)$ from the above proposition is a Poisson map. 
\end{prop}
A proof can be found in Appendix \ref{sums}.

As a result, we can construct canonical Darboux coordinates on $\mathcal C_{n,q}^0(m)$ by transferring them from $\mathcal C_{n,t}^0$. Namely, let us choose $A, B$ as suggested by \eqref{Eq:TadSigm_i}, \eqref{RSL}:
\begin{equation}\label{xysCyc}
A=\diag(x_1,\ldots, x_n)\,,\quad B_{ij}=\, \sigma_j \,\frac{t-1}{t-x_ix_j^{-1}}\,\prod_{k\neq j} \,\frac{1-tx_jx_k^{-1}}{1-x_jx_k^{-1}}\,.
\end{equation}
Then \eqref{abcm} determines $\widetilde V, \widetilde W$ (uniquely, up to a simultaneous rescaling). By mapping these $A,B, \widetilde V, \widetilde W$ to $X_s, Z_s, V, W$ as described in Proposition \ref{isocm}, we obtain local coordinates $x=(x_1, \dots, x_n)$, $\sigma=(\sigma_1, \dots, \sigma_n)$ on $\mathcal C_{n,q}^0(m)$. Then the results of Section \ref{tadpole} combined with Proposition \ref{IsoPssCyc} tell us that 
\begin{equation*}
\{x_i,x_j\}=0 \, , \quad
\{x_i,\sigma_j\}=\delta_{ij} x_i \sigma_j\, , \quad
 \{\sigma_i,\sigma_j\}=0\,. 
 \end{equation*}

\medskip

\begin{rem} By \cite[Theorem 5.17]{BEF}, the Calogero--Moser space $\mathcal C_{n,q}(m)$ is connected. Therefore, $\mathcal C_{n,q}^0(m)$ is a dense open subset, so the above coordinates $x, \sigma$ can be used as local coordinates on both spaces. 
\end{rem}

\subsection{Generalised  Ruijsenaars--Schneider models}  \label{ss:CycRS}
Having obtained canonical coordinates on the Calogero--Moser space $\mathcal C_{n, q}(m)$, we can now turn to the Hamiltonians from Theorem \ref{Thm:cycInvol}.
The Hamiltonians $E_{j,m}:=\tr(X^{jm})=m\tr (X_0\dots X_{m-1})^j=m\tr A^j=m\sum_{i=1}^n x_i^j$ are trivial. Next, we have
\begin{equation*} 
F_{m,j}:=\tr(1+XY)^j=\sum_{i\in I} \tr(\Id_n+X_iY_i)^j=\sum_{i\in I} \tr(X_iZ_i)^j=\sum_{i\in I} (t_s)^j \tr B^j\,. 
\end{equation*}
Since $B$ is the Lax matrix for the Ruijsenaars--Schneider model, we do not find anything new here. Now let us look at
\begin{equation*} 
G_{m,j}:= \tr(Y+X^{-1})^{jm}= m \tr\, \left(Z_{m-1} \dots Z_0\right)^j=(t_0\dots t_{m-1})^j \,\tr(A^{-1}B^m)^j\,. 
\end{equation*}  
Ignoring the constant factor, the simplest Hamiltonian is $G_{m,1}:=\tr ( A^{-1} B^m)$, with $A, B$ given by \eqref{xysCyc}; this formula also makes sense for $m=0,1$.
Here are explicit formulas for $G_{m,1}$ with $m\le 3$, where we use the shorthand notation $\displaystyle{\Upsilon_{ij}\,=\,\frac{1-tx_ix_j^{-1}}{1-x_ix_j^{-1}}\,=\,\frac{t-x_jx_i^{-1}}{1-x_jx_i^{-1}}}$:
 \begin{subequations}
       \begin{align}
G_{0,1}=&\sum_{i=1}^n x_i^{-1}\,,\qquad G_{1,1}=\sum_{i=1}^n \sigma_i x_i^{-1}\prod_{a\ne i}\Upsilon_{ia}\,,
\\
G_{2,1}=& \sum_{i=1}^n \, \sigma_i^2 x_i^{-1}\,\prod_{a\neq i} \Upsilon_{ia}^2
\,+ \sum_{i<j}^n\, \sigma_i\sigma_j\frac{ (t-1)^2(x_i^{-1}+x_j^{-1}) } {(1-x_ix_j^{-1})(1-x_jx_i^{-1})}\,
\,\prod_{a\ne i, j} \Upsilon_{ia}\Upsilon_{ja}\,, \label{k21}\\
G_{3,1}=&\sum_{i=1}^n \, \sigma_i^3x_i^{-1}\,\prod_{a\neq i} \Upsilon_{ia}^3\,
+\sum_{i\neq j}^n\sigma_i\sigma_j^2\frac{(t-1)^2(x_i^{-1}+2x_j^{-1})}{(t-x_i x_j^{-1})(t-x_jx_i^{-1})}\,
\,\prod_{a\neq i} \Upsilon_{ia}\,\prod_{b\neq j} \Upsilon_{jb}^2 \nonumber \\
&+\sum_{i\neq j\neq k} \, \sigma_i\sigma_j\sigma_k\frac{(t-1)^3x_i^{-1}}
{(t-x_i x_j^{-1})(t-x_jx_k^{-1})(t-x_k x_i^{-1})}\,
\,\prod_{a\neq i} \Upsilon_{ia}\,\prod_{b\neq j} \Upsilon_{jb}\,\prod_{c\neq k} \Upsilon_{kc}\,.
       \end{align}
  \end{subequations}
The general formula is
\begin{equation}\label{gexp}
G_{m,1}=\sum_{1\le j_0,\dots, j_{m-1}\le n} \, (\sigma_{j_0}\dots \sigma_{j_{m-1}})x_{j_0}^{-1}\prod_{s=0}^{m-1}\frac{t-1}
{t-x_{j_s} x_{j_{s+1}}^{-1}}\,\prod_{s=0}^{m-1}
\,\prod_{a\neq j_s}^n \Upsilon_{j_sa}\,.
\end{equation} 
Finally, let us look at
\begin{equation*}
H_{m, j}:= \tr Y^{jm}=m \tr (Y_{m-1}\dots Y_0)^j\,.
\end{equation*}
We have $Y_s=Z_s-X_s^{-1}$, so in terms of $A,B$ we have
\begin{equation*}
Y_s=t_sB-\Id_n\ (s=0,\dots, m-2)\,,\qquad Y_{m-1}=A^{-1}(t_{m-1}B-\Id_n)\,. 
\end{equation*}
After rescaling,
\begin{equation*}
H_{m, j}=\tr(A^{-1} P(B))^j\,,\quad\text{where}\ \ P(B):=\prod_{i\in I} (B-t_i^{-1}\Id_n)\,.
\end{equation*}
We see that in the limit $t_i\to\infty$, each $H_{m,j}$ tends to $G_{j,m}=\tr (A^{-1}B^m)^j$. Thus, $\{H_{m,j}\}$  is a more general $m$-parametric family of integrable systems. For a given $m$, the simplest Hamiltonian is $H_{m,1}=\tr (A^{-1}P(B))$, which is nothing but a general linear combination of $G_{m,1}$ with smaller $m$. For example,
\begin{equation}\label{h21}
H_{2,1}=\tr(A^{-1}B^2)-(t_0^{-1}+t_1^{-1})\tr(A^{-1}B)+(t_0t_1)^{-1}\tr A^{-1} 
\end{equation}
is a linear combination of $G_{2,1}$, $G_{1,1}$ and $G_{0,1}$. Therefore, $H_{m,1}$ can be written explicitly using the expressions for $G_{l,1}$ with $l\le m$.


\section{Quantization and further links} \label{final} 

Integrable systems closely related to those constructed above appeared recently in a different context, so below we indicate these connections.

\subsection{Twisted Macdonald--Ruijsenaars systems}

In \cite[Appendix]{CE}, certain generalisations of the quantum Macdonald--Ruijsenaars system were proposed. Such generalisations depend on an integer $\ell\ge 2$ and they exist for any root system $R$. Their construction given in \cite{CE} is very implicit: at the first step, the eigenfunctions $\psi_\ell(\lambda, z)$ for the twisted system are constructed by integrating certain products of the Gaussian and the eigenfunctions $\psi(\lambda, z)$ of the usual Macdonald--Ruijsenaars system. By analysing the properties of $\psi_\ell(\lambda, z)$, the existence of a complete family of commuting quantum Hamiltonians is then deduced (see \cite[Appendix]{CE} for more details). To compare these integrable systems with the ones constructed above, let us write the corresponding quantum Hamiltonian explicitly in the case $R=A_{n-1}$ and $\ell=2$. 

We consider the algebra of difference operators in $n$ variables $z_1, \dots, z_n$ and denote by $T_i$ the shift operator in the $i$th variable, acting by $T_if(z_1,\dots, z_n)=f(z_1, \dots, z_i+1, \dots, z_n)$. 
It will be convenient to introduce $q\in\CC^{\times}$, not a root of unity, and work with exponential coordinates $x_i=q^{z_i}$, also allowing $x_i^{1/2}=q^{z_i/2}$. We have $T_i(x_i^{1/2})=q^{1/2}x_i^{1/2}$. 
In the formulas below we will aslo use $t\ne 0$ as a coupling parameter of the system.

\begin{prop}\label{tmr} For $R=A_{n-1}$ and $\ell=2$, the Hamiltonian of the twisted Macdonald--Ruijsenaars system \cite{CE} is given by the following difference operator $D_{2,1}$:
\begin{equation}
D_{2,1}=\sum_{i=1}^n a_i T_i^2+\sum_{i<j}^n b_{ij} T_iT_j\,,
\end{equation}
where the coefficients $a_i$, $b_{ij}$ are given by
\begin{align}
a_i=&\prod_{j\ne i}^n \frac{(1-tx_ix_j^{-1})(1-qtx_ix_j^{-1})}{(1-x_ix_j^{-1})(1-qx_ix_j^{-1})}\,,\\
b_{ij}=&q^{1/2}(t-1)(t-q)\frac{(x_i^{1/2}x_j^{-1/2}+x_i^{-1/2}x_j^{1/2})}{(1-qx_ix_j^{-1})(1-qx_jx_i^{-1})}
\prod_{l\ne i,j}^n \frac{(1-tx_ix_l^{-1})(1-tx_jx_l^{-1})}{(1-x_ix_l^{-1})(1-x_jx_l^{-1})}\,.
\end{align}
\end{prop}

\begin{proof}
According to \cite[Theorem 7.1(1)]{CE}, for any (reduced) root system $R\subset V$ of a Weyl group $W$, $W$-invariant multiplicities $m_\alpha\in\Z_{\ge 0}$, and any $\ell\in\N$, there exists a twisted BA function $\psil(\lambda, z)$ of the form  
\begin{equation}\label{psil}
\psil(\lambda,z) = q^{\langle\lambda, z\rangle/\ell}\sum_{\nu\in\nn\cap \ell^{-1}P} \psi_\nu(\lambda) q^{\langle\nu,z\rangle}\,,\quad \lambda, z\in V\,,
\end{equation}
where $\nn$ denotes the convex hull of the $W$-orbit of a vector $\rho=\sum_{\alpha\in R_+}m_\alpha\alpha$, and $P$ is the weight lattice of $R$. The function $\psil$ is characterised by the following properties: for
each $\alpha\in R$, $j=1,\dots ,m_\alpha$ and any $\epsilon$ with $\epsilon^\ell=1$ one has
\begin{equation}\label{axpsil}
\psil\left(\lambda,\, z-{\frac 12 j\alpha}\right) = \epsilon^{j}\psil\left(\lambda,
\,z{+\frac 12 j\alpha}\right)\quad\text{for}\ \
q^{\langle\alpha, z\rangle/\ell}=\epsilon\,.
\end{equation}
Properties \eqref{axpsil} determine $\psil$ uniquely, up to an arbitrary $\lambda$-dependent factor. Moreover, by \cite[Theorem 7.1(3)]{CE}, the function $\psil$ is a common eigenfunction of a family of commuting $W$-invariant difference operators $D_{\ell}^{\pi}$ in the $x$-variable, so that $D_{\ell}^\pi\psil=\mathfrak m_\pi(\lambda)\psil$. Here $\pi$ is any dominant weight and $\mathfrak m_\pi(\lambda)=\sum_{\tau\in W\pi}q^{\langle\tau, \lambda\rangle}$ is the corresponding orbitsum.   

Let us consider the case $R=A_{n-1}$, $m_\alpha=m$ and $\ell=2$. In this case properties \eqref{axpsil} can be rewritten as
\begin{equation}\label{axpsila}
\psil(\lambda,\, z+je_a) = \epsilon^{j}\psil(\lambda,
\,z+ je_b)\quad\text{for}\ \
q^{z_a/2}=\epsilon q^{z_b/2}\,,
\end{equation}
which should hold for all $j=1,\dots, m$, distinct $a,b\in\{1,\dots n\}$ and $\epsilon=\pm 1$.

To identify the operator $D_{2,1}$ with one of the twisted Macdonald--Ruijsenaars Hamiltonian from \cite{CE}, it suffices to show that 
\begin{equation}\label{ei}
D_{2,1}\psil(\lambda, z)=\left(\sum_{i=1}^n q^{\lambda_i}\right)\psil(\lambda, z)\,.
\end{equation}
Using the approach of \cite[Section 3]{C}, this would follow once we establish that the operator $D_{2,1}$ preserves the properties \eqref{axpsila} of $\psil$. Note that if we multiply $\psil$ by the function $g(z)=q^{\langle z, z\rangle/4}$, where $\langle z, z \rangle=z_1^2+\dots+z_n^2$, then $\widetilde\psil(\lambda, z)=g(z)\psil(\lambda, z)$ will satisfy the following conditions:
\begin{equation}\label{axpsilb}
\widetilde\psil(\lambda,\, z+je_a) = \widetilde\psil(\lambda,
\,z+ je_b)\quad\text{for}\ \
q^{z_a/2}=\epsilon q^{z_b/2}\,,
\end{equation}
which means that 
\begin{equation}\label{axpsilc}
\widetilde\psil(\lambda,\, z+je_a) = \widetilde\psil(\lambda,
\,z+ je_b)\quad\text{for}\ \
q^{z_a}=q^{z_b}\,.
\end{equation}
Therefore, it remains to check that the operator $\widetilde D_{2,1}:=g\circ D_{2,1}\circ g^{-1}$ in the case $t=q^{-m}$ preserves the properties \eqref{axpsilc} for all $j=1,\dots, m$.
Explicitly, we have
\begin{equation}\label{d21}
\widetilde D_{2,1}=\sum_{i=1}^n\widetilde a_i T_i^2+\sum_{i<j}^n \widetilde b_{ij} T_iT_j\,,
\end{equation}
where the coefficients $\widetilde a_i$, $\widetilde b_{ij}$ are given by
\begin{align}\label{d21a}
\widetilde a_i=&(qx_i)^{-1}\prod_{j\ne i}^n \frac{(1-tx_ix_j^{-1})(1-qtx_ix_j^{-1})}{(1-x_ix_j^{-1})(1-qx_ix_j^{-1})}\,,\\
\label{d21b}
\widetilde b_{ij}=&\frac{(t-1)(t-q)(x_i^{-1}+x_j^{-1})}{(1-qx_ix_j^{-1})(1-qx_jx_i^{-1})}
\prod_{l\ne i,j}^n \frac{(1-tx_ix_l^{-1})(1-tx_jx_l^{-1})}{(1-x_ix_l^{-1})(1-x_jx_l^{-1})}\,.
\end{align} 
For this operator we have the following result.

\begin{lem}
Let $\mathcal Q_m$ denote the space of functions $f(x_1, \dots, x_n)$ holomorphic on $(\CC^\times)^n$ and such that for any $j=1, \dots, m$ and $1\le a<b\le n$ we have $(T_a)^j f=(T_b)^j f$ for $x_a=x_b$. Then $ \widetilde D_{2,1}(\mathcal Q_m)\subseteq \mathcal Q_m$. 
\end{lem}     
This lemma is proved analogoulsy to \cite[Proposition 2.1]{C},  using \cite[Lemma 2.5]{C}. It implies that the operator $\widetilde D_{2,1}$ preserves the properties \eqref{axpsilc}, so we are done.
\end{proof}

To see a link with the Hamiltonian $G_{2,1}$ \eqref{k21}, consider the classical limit of $\widetilde D_{2,1}$. On the quantum level we have the algebra of $q$-difference operators, with $[T_i,x_j]=\delta_{ij} (q-1)x_iT_i$. In the classical limit $q=e^{-\hbar}\to 1$ we obtain $2n$ commuting variables $\bar x_i^{\pm 1}, \bar T_i^{\pm 1}$ with the Poisson bracket $\{\bar x_i, \bar T_j\}=\delta_{ij}\bar x_i\bar T_i$. The classical limit of $\widetilde D_{2,1}$ is, therefore, the following function $\bar D_{2,1}$: 
\begin{align}
\bar D_{2,1}=&\sum_{i=1}^n\bar a_i \bar T_i^2+\sum_{i<j}^n \bar b_{ij} \bar T_i\bar T_j\,,\quad
\bar a_i=\bar x_i^{-1}\prod_{j\ne i}^n \frac{(1-t\bar x_i\bar x_j^{-1})^2}{(1-\bar x_i\bar x_j^{-1})^2}\,,\\
\bar b_{ij}=&\frac{(t-1)^2(\bar x_i^{-1}+\bar x_j^{-1})}{(1-\bar x_i\bar x_j^{-1})(1-\bar x_j\bar x_i^{-1})}
\prod_{l\ne i,j}^n \frac{(1-t\bar x_i\bar x_l^{-1})(1-t\bar x_j\bar x_l^{-1})}{(1-\bar x_i\bar x_l^{-1})(1-\bar x_j\bar x_l^{-1})}\,.
\end{align}
A substitution $\bar x_i=x_i$, $\bar T_i=\sigma_i$ matches it to the formula \eqref{k21}. 

\begin{rem} 
More generally, it follows from the results of \cite{BEF} that for any $\ell$, the twisted Macdonald--Ruijsenaars system from \cite[Appendix]{CE} in type $A$ coincides, in the classical limit, with the system defined by $G_{\ell, j}=\tr (A^{-1}B^\ell)^j$, $j\in\N$. See \cite[Remark 3.26]{BEF}.     
\end{rem}

\begin{rem}
According to \cite[Theorem 7.3]{CE}, for $t=q^{-m}$ with $m\in\Z_+$ the operator $D_{2,1}$ is algebraically integrable. Moreover, in that case its eigenfunctions $\psil(\lambda, z)$ are given by \cite[(7.2)]{CE}, with $\ell=2$. The operator $D_{2,1}$ is bispectrally self-dual, namely,  one has $\psil(\lambda, z)=\psil(z, \lambda)$.
\end{rem}

\subsection{Cyclotomic DAHA and quiver quage theory}

In the process of writing this paper we became aware of the work of Braverman, Finkelberg and Nakajima \cite{BFNa, BFNb}, and of Kodera and Nakajima \cite{KN}, where some operators generalising the Macdonald--Ruijsenaars operators appeared in the context of quiver gauge theory. In particular, in \cite{KN} an isomorphism is established between the quantized Coulomb branch of a $3d$ $\mathcal N=4$ supersymetric gauge theory and the spherical subalgebra of the Cherednik algebra for $W=\Z_l\wr S_N$. The form of this isomorphism (see \cite[Theorem 1.5]{KN}) suggested to us that there should be a relation to the integrable systems constructed in this paper. Indeed, on the classical level this can be seen as follows. Recall that for each $m$, the Calogero--Moser space $\mathcal C_{n, q}(m)$ carries three families of Poisson-commuting functions $E_{m,j}=\tr X^{mj}$, $F_{m,j}=\tr (1+XY)^{j}$, and $H_{m,j}=\tr Y^{jm}$. We have an open dense subspace of  $\mathcal C_{n, q}(m)$ where $X$ is invertible and everything can be written in terms of the matrices $A,B$ that parametrize the Calogero--Moser space $\mathcal C_{n, t}^0$ for the tadpole quiver, resulting in 
\begin{equation}\label{efh}
E_{m,j}=\tr A^j\,,\quad F_{m,j}= \tr B^j\,,\quad H_{m,j}=\tr (A^{-1}\prod_{i\in I} (B-t_i^{-1}\Id_n))^j\,.  
\end{equation}
Choosing the coordinates $x, \sigma$ as in \eqref{xysCyc}, we obtain the expressions for these Hamiltonians as in section \ref{ss:CycRS}. That coordinate system was chosen so to make $A$ diagonal. Instead, we can choose to diagonalise $B$, parametrizing $A,B$ by $w=(w_1, \dots, w_n)$ and $u=(u_1, \dots, u_n)$ as follows:
\begin{equation}\label{RSLcyc}
B=\diag(w_1, \dots, w_n)\,,\quad  A_{ij}=\, u_j\,\frac{t^{-1}-1}{t^{-1}-w_iw_j^{-1}}\,\prod_{k\neq j} \,\frac{1-t^{-1}w_jw_k^{-1}}{1-w_jw_k^{-1}}\,.
\end{equation}  
This corresponds to the canonical transformation $\Phi: (A,B)\mapsto (B,A)$ (which changes the sign of the bracket), so we have:
\begin{equation*}
\{w_i,w_j\}=0 \, , \quad
\{u_i, w_j\}=\delta_{ij} u_i w_j\, , \quad
 \{u_i,u_j\}=0\,. 
 \end{equation*}   
To write the Hamiltonians \eqref{efh} in these new coordinates, one needs to calculate $A^{-1}$. This is done with the help of the Cauchy's determinant formula; the result is
\begin{equation*}
(A^{-1})_{ij}=u_i^{-1} \frac{t-1}{t-w_iw_j^{-1}}\,\prod_{k\neq j} \,\frac{1-tw_jw_k^{-1}}{1-w_jw_k^{-1}}\,.
\end{equation*}
As a result, we have
\begin{equation}\label{hd}
E_{m,1}=\sum_{i=1}^n \prod_{j\neq i} \,\frac{1-t^{-1}w_iw_j^{-1}}{1-w_iw_j^{-1}} u_i\,,\quad F_{m,j}=\sum_{i=1}^n w_i^j\,,\quad 
H_{m,1}=\sum_{i=1}^n \prod_{j\neq i} \,\frac{1-tw_iw_j^{-1}}{1-w_iw_j^{-1}} \prod_{k=0}^{m-1} (w_i-t_k^{-1}) u_i^{-1}\,. 
\end{equation}
Note that in this form the integrability of the Hamiltonians $E_{m,1}$ and $H_{m,1}$ becomes obvious: $E_{m,1}$ {\it is} the Ruijsennars--Schneider Hamiltonian, and $H_{m,1}$ reduces to such by a canonical transformation $w_i\mapsto w_i$, $u_i\mapsto \prod_{k=0}^{m-1} (w_i-t_k^{-1}) u_i$ ($i=1,\dots, n$). By contrast, in coordinates $x_i, \sigma_i$ as in Section \ref{ss:CycRS}, the Hamiltonian $H_{m,1}$ looks complicated and its integrability is not obvious without knowing that change of variables that reduces it to the Ruijsenaars--Schneider system. We should emphasize that the Hamiltonian flows defined by $H_{m,j}$ are {\it not} complete when viewed on $\mathcal C_{n,q}^0(m)\simeq \mathcal C_{n,t}^0$, so one does need a cyclic-quiver interpretation in order to get a completed phase space and integrate the flows.

Now, there is an obvious parallel between \eqref{hd} and the generators $E_1[1]$, $F_1[1]$ and $\sum_{i} w_i^j$ of the quantized Coulomb branch as in \cite[Theorem 1.5]{KN}. Therefore, one should expect the above Hamiltonians to be connnected with the $K$-theoretic Coulomb branch of a quiver gauge theory, cf. \cite[A(ii) and Remark A.6]{BFNb}. Then the quantized $K$-theoretic Coulomb branch \cite{BFNa, BFNb} should also provide quantization of these integrable Hamiltonians. 

In fact, a recent paper of Braverman, Etingof and Finkelberg (with an appendix by Nakajima and Yamakawa) \cite{BEF} greatly clarifies this connection. In particular, it introduces a cyclotomic version of the double affine Hecke algebra for $\Gl_n$, giving a construction of both the quantum and the classical versions of the integrable systems considered in the present paper. Below we indicate the relationship between our results and those from \cite{BEF}. 

The Calogero--Moser space $\mathcal C_{n, q}(m)$ is the same as the space $\mathcal M_N^l(Z,t)$ from \cite[Section 5]{BEF} with $l=m$ and $N=n$. Our parameter $t$ is the same as in \cite{BEF}, while their parameters $Z_i$ are related to our $q_i$ by $Z_{i-1}/Z_{i}=q_i$. By \cite[Theorems 5.16 \& 5.17]{BEF},  the coordinate ring $\mathcal O(\mathcal M_N^l(Z,t))$ is identified with the spherical subalgebra  
of the cyclotomic DAHA $H\!\!H^l_N(Z,1,t)$. This ring admits a noncommutative deformation, a spherical subalgebra $e_N H\!\!H^l_N(Z,q,t) e_N$, where $q$ is a quantization parameter. The cyclotomic DAHA $H\!\!H^l_N(Z,q,t)$, in its turn, is defined as a subalgebra of the usual DAHA for $\Gl_N$, see \cite[Sec. 3.4]{BEF}. By \cite[Theorem 3.28]{BEF}, the algebra $H\!\!H^l_N(Z,q,t)$ contains three commutative subalgebras generated by $X_i$, $Y_i^{\pm 1}$ and $D_i^{(l)}$ with $i=1,\dots ,N$. By symmetrization, one obtains three commutative subalgebras in the spherical subalgebra of $H\!\!H^l_N(Z,q,t)$, see \cite[Section 3.6]{BEF}. The elements of the spherical subalgebra can be realized as difference operators; in this way one obtains three commuting subalgebras of difference operators that quantize the Poisson-commuting families $\{E_{l,j}\,|\,j\in\N\}$, $\{F_{l,j}\,|\,j\in\Z\}$ and $\{H_{l,j}\,|\, j\in\N\}$ from Section \ref{ss:CycRS} (the first family consists of functions of $x_i$ so its quantization is obvious). 

Let us note that writing down  these quantum integrable Hamiltonians expliciltly does not seem easy. In the case of a quiver with two vertices,  we have the following explicit formula for the quantization of $H_{2,1}$ \eqref{h21}.

\begin{prop} The quantum Hamiltonian 
\begin{equation}
\widetilde{H_{2,1}}=\sum_{i=1}^n\widetilde a_i T_i^2+\sum_{i<j}^n \widetilde b_{ij} T_iT_j+ \alpha\sum_{i=1}^n \prod_{k\ne i}^n\frac{1-tx_ix_k^{-1}}{1-x_ix_k^{-1}}x_i^{-1}T_i+\beta\sum_{i=1}^n x_i^{-1}
\end{equation}
is completely integrable for any values of the parameters $\alpha, \beta$. Here the coefficients $\widetilde a_i$, $\widetilde b_{ij}$ are given by \eqref{d21a}--\eqref{d21b}. 
\end{prop}
 
This formula is obtained by combining Proposition \ref{tmr}, \eqref{d21} and \cite[Corollary 3.22, Example 3.24, Remarks 3.25 \&  3.26]{BEF}. \qed
    
There is also a ``dual'' realization of these quantum Hamiltonians, obtained by applying the Cherednik--Fourier transform, see \cite[Section 3.6]{BEF}. These dual families are quantum versions of \eqref{hd}. Essentially, they are given by the same formulas as in \eqref{hd}, but with $u_i$ replaced by the multiplicative shift $T_i$.

\appendix

\section{Calculations with the brackets} \label{Ann:brackets} 

In this section we prove Propositions \ref{Prop:int}, \ref{Prop:dBrTad}, \ref{Prop:intcyclic}, \ref{Prop:dBrcyclic}. We will start with the case of the cyclic quiver. Given $m\ge 2$, we set $I:=\Z/m\Z$. For $r\in\Z$, we set 
\begin{equation}\label{er}
E_r:=\sum_{i\in I} e_{i+r}\otimes e_i\in A\otimes A\,.
\end{equation} 
Recall that $x:=\sum_{i\in I}x_i$, $y:=\sum_{i\in I} y_i$, with $e_ix=xe_{i+1}$,  $e_{i+1}y=ye_i$. The element $e=\sum_i e_i$ commutes with $x,y$, so we will identify it with the identity. If one introduces a $\Z$-grading on $k\langle x, y\rangle$ by setting $\deg\,x=1$, $\deg\,y=-1$, then
\begin{equation}\label{gr}
(uE_rv)^\circ=vE_{r'}u\,,\quad r'=-r+\deg u+\deg v\,,
\end{equation}
for any homogeneous elements $u,v\in k\langle x, y\rangle$.
  
\begin{lem}
We have:
\begin{equation}\label{xyxy}
\dgal{xy, xy}=xyE_0-E_0xy-\frac12 E_0(xy)^2+\frac12 (xy)^2E_0\,.
\end{equation}
\end{lem}
\begin{proof}
\begin{equation}\label{lb}
\dgal{xy, xy}= \dgal{xy, x}y+x\dgal{xy, y}=-\dgal{x,xy}^\circ y-x\dgal{y,xy}^\circ\,. 
\end{equation}
Using the formulas \eqref{cycida}--\eqref{cycidc}, we have
\begin{equation*}
\dgal{x,xy}=\dgal{x,x}y+x\dgal{x,y}=\frac12(E_1x^2-x^2E_1)y+x\left(E_{1}+\frac12(yxE_{1}+E_{1}xy-yE_{1}x+xE_{1}y)\right)\,,
\end{equation*}
and so 
\begin{equation*}
\dgal{x,xy}^\circ=\frac12(x^2yE_{0}-yE_{0}x^2)+E_0x+\frac12(E_0xyx+xyE_0x-xE_0xy+yE_0x^2)\,.
\end{equation*}
Similarly, we find that
\begin{eqnarray*}
\dgal{y,xy}&=&\frac12(xy^2E_{-1}-xE_{-1}y^2)-E_{-1}y-\frac12(E_{-1}yxy+xyE_{-1}y-xE_{-1}y^2+yE_{-1}xy)\,,\\
\dgal{y,xy}^\circ &=&\frac12(E_0xy^2-y^2E_0x)-yE_0-\frac12(yxyE_0+yE_0xy-y^2E_0x+xyE_0y)\,.
\end{eqnarray*}
Substituting these expressions into \eqref{lb} leads, after cancellations, to the expression \eqref{xyxy}.
\end{proof}

\begin{lem}
Let us adjoin to $A$ the elements $x_i^{-1}$, so that $x^{-1}=\sum_{i\in I} x_i^{-1}$. Then for $z=y+x^{-1}$ one has:
\begin{equation}\label{zz}
\dgal{z, z}=\frac12 (z^2E_{-1}-E_{-1}z^2)\,.
\end{equation}
\end{lem}
\begin{proof}
This formula can be checked directly, using that $0=\dgal{y, xx^{-1}}=\dgal{y,x}x^{-1}+x\dgal{y,x^{-1}}$ and $0=\dgal{x, xx^{-1}}=\dgal{x,x}x^{-1}+x\dgal{x,x^{-1}}$. Alternatively, we can use the idea from the proof of Theorem \ref{Thm:QStructloc}. Namely, one can rewrite the bivector $\PP$ in terms of the generators $x_i, z_i=y_i+x_i^{-1}$. The resulting double brackets are almost identical (with $y_i$ replaced by $z_i$, which is now thought of as $x_i^*$), the only difference appears in the brackets $\dgal{a, a^*}$, see \eqref{aastloc}. Therefore, we have
 \begin{subequations}
       \begin{align}
\dgal{x,x}\,=\,&\frac{1}{2}(E_1x^2- x^2E_1)\,,\quad
\dgal{z,z}\,=\,\frac{1}{2}(z^2E_{-1}- E_{-1}z^2)\,,\\
\dgal{x,z}\,=\,&\frac{1}{2} (zxE_1+E_1xz-zE_1x+xE_1z)\,.\label{xz}
\end{align}
  \end{subequations} 
This gives the needed formula for $\dgal{z,z}$.
\end{proof}

Now we need the following general lemma. Let $A$ be an associative $k$-algebra with a double bracket $\dgal{-,-}$ and the associated ordinary bracket $\{-,-\}=m\circ \dgal{-,-}$. 

\begin{lem} Given $a\in A$, suppose that $\mathcal E\subset A\otimes A$ is a subset such that for any $E\in \mathcal E$ and any $r,s\ge 0$ we have $(a^rEa^s)^\circ=a^sE'a^r$ for some $E'\in\mathcal E$ which depends on $r+s$ but not on $r,s$ individually. Moreover, assume that $a$ commutes with any element in $m(\mathcal E)\subset A$. If $\dgal{a,a}=\sum_{i} (a^iE_i-E_i a^i)$ with $E_i\in\mathcal E$, then $\{a^k, a^l\}=0$ for all $k,l$.
\end{lem} 

\begin{proof}
Since $\{-,-\}$ satisfies Leibniz's rule in the second argument, it is enough to prove that $\{a^k, a\}=0$ for all $a$. Now, $\dgal{a^k, a}=-\dgal{a, a^k}^{\circ}$, while
\begin{equation*}
\dgal{a, a^k}=\sum_{\genfrac{}{}{0pt}{2}{r, s\ge 0}{r+s=k-1}} a^{r}\dgal{a,a}a^{s}=\sum_{\genfrac{}{}{0pt}{2}{r, s\ge 0}{r+s=k-1}} \sum_{i} a^{r}(a^i E_i- E_i a^i)a^{s}\,,
\end{equation*} 
for some $E_i\in\mathcal E$. Using that $(a^{r+i} E_i a^s-a^rE_ia^{s+i})^\circ =a^s E_i' a^{r+i}-a^{s+i}E_i'a^r$, for some $E_i'\in\mathcal E$, we obtain:
\begin{equation*}
m (\dgal{a^k, a})=-\sum_{\genfrac{}{}{0pt}{2}{r, s\ge 0}{r+s=k-1}} \sum_{i} m(a^{s} E_i' a^{r+i}- a^{s+i} E_i' a^{r})=\sum_{r,s}\sum_{i}\left(a^s m(E_i') a^{r+i}-a^{s+i}m(E_i')a^r\right)=0\,,
\end{equation*}
since $m(E_i')$ commutes with $a$. Therefore, $\{a^k,a\}=0$, as needed. 
\end{proof}

\begin{proof}[Proof of Proposition \ref{Prop:intcyclic}.] 
We now use this lemma with $\mathcal E=\oplus_{r\in I} \CC E_r$ with $E_r$ as in \eqref{er}, and with $a=x, y, xy$ or $z=y+x^{-1}$. We have $m(\mathcal E)=\CC e$, where $e:=\sum_{i}{e_i}$ commutes with any of the above $a$. The assumptions of the lemma are satisfied due to \eqref{gr}. The rest follows from \eqref{cycida}, \eqref{xyxy} and \eqref{zz}.  
\end{proof}


\begin{proof}[Proof of Proposition \ref{Prop:dBrcyclic}.] 
We start by calculating some double brackets. First,
\begin{eqnarray*}
\dgal{x,x^a}&=&\sum_{\genfrac{}{}{0pt}{2}{r, s\ge 0}{r+s=a-1}}x^r\dgal{x,x}x^s=\sum_{\genfrac{}{}{0pt}{2}{r, s\ge 0}{r+s=a-1}}\left(\frac12 x^{r}E_1x^{s+2}-\frac12 x^{r+2}E_1x^{s}\right)\,,\\
\dgal{x^a,x}&=&-\dgal{x,x^a}^\circ=\sum_{\genfrac{}{}{0pt}{2}{r, s\ge 0}{r+s=a-1}}\left(-\frac12 x^{s+2}E_{a}x^{r}+\frac12 x^{s}E_ax^{r+2}\right)\,,\\
\dgal{x^a,x^b}&=&\sum_{\genfrac{}{}{0pt}{2}{r, s\ge 0}{r+s=b-1}}x^r\dgal{x^a,x}x^s\\
&=&\sum_{\genfrac{}{}{0pt}{2}{r, s\ge 0}{r+s=b-1}}\sum_{\genfrac{}{}{0pt}{2}{r', s' \ge 0}{r'+s'=a-1}}\left(-\frac12 x^{r+s'+2}E_ax^{r'+s}+\frac12 x^{r+s'}E_ax^{r'+s+2}\right)\,,\\
&=&\sum_{\genfrac{}{}{0pt}{2}{r, s\ge 0}{r+s=b-1}}\sum_{\genfrac{}{}{0pt}{2}{r', s' \ge 0}{r'+s'=a-1}}\left(-\frac12 x^{r+s'+2}E_ax^{r'+s}+\frac12 x^{r'+s}E_ax^{r+s'+2}\right)\,.
\end{eqnarray*}
Next, we have 
\begin{eqnarray*}
\dgal{y, x^a}\,&=&\, \sum_{\genfrac{}{}{0pt}{2}{r, s\ge 0}{r+s=a-1}} x^{r}\dgal{y,x} x^{s}\\
&=&\sum_{\genfrac{}{}{0pt}{2}{r, s\ge 0}{r+s=a-1}}\left( -x^{r}E_{-1}x^{s}-\frac12 x^{r}E_{-1}yx^{s+1} -\frac12 x^{r+1}yE_{-1}x^{s} +\frac12 x^{r+1}E_{-1}yx^{s} -\frac12 x^{r}yE_{-1}x^{s+1} \right)\,.
\end{eqnarray*}
For $\dgal{x^a, y}=-\dgal{y, x^a}^\circ$ we, therefore, obtain: 
\begin{equation*}
\dgal{x^a, y}=\sum_{\genfrac{}{}{0pt}{2}{r, s\ge 0}{r+s=a-1}}\left( x^{s}E_{a}x^{r}+\frac12 yx^{s+1}E_ax^{r} +\frac12 x^{s}E_ax^{r+1}y -\frac12 yx^{s}E_ax^{r+1} +\frac12 x^{s+1}E_ax^{r}y\right)\,.
\end{equation*}  
From this,
\begin{equation*}
\begin{aligned}
\dgal{x^a, yx^b}&= \dgal{x^a,y}x^b+y\dgal{x^a, x^b}\\
&=\sum_{\genfrac{}{}{0pt}{2}{r, s\ge 0}{r+s=a-1}}\left( x^{s}E_ax^{r+b}+\frac12 yx^{s+1}E_ax^{r+b} +\frac12 x^{s}E_ax^{r+1}yx^b -\frac12 yx^{s}E_ax^{r+1+b} +\frac12 x^{s+1}E_ax^{r}yx^b\right)\\
&+\sum_{\genfrac{}{}{0pt}{2}{r, s\ge 0}{r+s=b-1}}\sum_{\genfrac{}{}{0pt}{2}{r', s'\ge 0}{r'+s'=a-1}} \left(-\frac12 yx^{r+s'+2}E_ax^{r'+s}+\frac12 yx^{r'+s}E_ax^{r+s'+2}\right)\,.
\end{aligned}
\end{equation*} 
As a result, for $\dgal{yx^b, x^a}=-\dgal{x^a, yx^b}^\circ$ we get:
\begin{equation*}
\begin{aligned}
\dgal{yx^b, x^a}&= \sum_{\genfrac{}{}{0pt}{2}{r, s\ge 0}{r+s=a-1}}( -x^{r+b}E_{b-1}x^{s}-\frac12 x^{r+b}E_{b-1}yx^{s+1} -\frac12 x^{r+1}yx^bE_{b-1}x^{s} +\frac12 x^{r+1+b}E_{b-1}yx^{s} -\frac12 x^{r}yx^bE_{b-1}x^{s+1})\nonumber\\
&+\sum_{\genfrac{}{}{0pt}{2}{r, s\ge 0}{r+s=b-1}}\sum_{\genfrac{}{}{0pt}{2}{r', s'\ge 0}{r'+s'=a-1}} \left(\frac12 x^{r'+s}E_{b-1}yx^{r+s'+2}-\frac12 x^{r+s'+2}E_{b-1}yx^{r'+s}\right)\,.
\end{aligned}
\end{equation*}  
Now we calculate
\begin{equation*}
\begin{aligned}
\dgal{y, yx^a}\,&=\dgal{y,y}x^a+y\dgal{y,x^a}\\
&=\frac12 y^2E_{-1}x^a-\frac12 E_{-1}y^2x^a\\
&+\sum_{\genfrac{}{}{0pt}{2}{r, s\ge 0}{r+s=a-1}} \left( -yx^{r}E_{-1}x^{s}-\frac12 yx^{r}E_{-1}yx^{s+1} -\frac12 yx^{r+1}yE_{-1}x^{s} +\frac12 yx^{r+1}E_{-1}yx^{s} -\frac12 yx^{r}yE_{-1}x^{s+1} \right)\,.
\end{aligned}
\end{equation*}
Therefore, for $\dgal{yx^a, y}=-\dgal{yx^a, y}^\circ$ we obtain:
\begin{equation*}
\begin{aligned}
\dgal{yx^a, y}\,&= -\frac12 x^aE_{a-1}y^2+\frac12 y^2x^aE_{a-1}\\
&+\sum_{\genfrac{}{}{0pt}{2}{r, s\ge 0}{r+s=a-1}} ( x^sE_{a-1}yx^{r}+\frac12 yx^{s+1}E_{a-1}yx^{r} +\frac12 x^sE_{a-1}yx^{r+1}y -\frac12 yx^sE_{a-1}yx^{r+1} +\frac12 x^{s+1}E_{a-1}yx^{r}y)\,.
\end{aligned}
\end{equation*}

All this allows us to find the ordinary brackets $\{u,v\}=\dgal{u,v}'\dgal{u,v}''$. First, we have for $a=0\mod m$:
\begin{equation}\label{xay}
\{x^a, y\}=\sum_{\genfrac{}{}{0pt}{2}{r, s\ge 0}{r+s=a-1}}( x^{s+r}+ x^{s+1+r}y)=ax^{a-1}+ax^ay\,,
\end{equation}
from which it follows that
\begin{equation*}
\{x^a, yx^b\}=\{x^a, y\}x^b=ax^{a+b-1}+ax^{a}yx^{b}=ax^{a+b-1}+ayx^{a+b}\mod [A,A]\,,
\end{equation*}
which is the first relation in Proposition \ref{Prop:dBrcyclic}.

Next, for $b=1\mod m$ 
\begin{eqnarray*}
\{yx^b, x^a\}&=& \sum_{\genfrac{}{}{0pt}{2}{r, s\ge 0}{r+s=a-1}}\left( -x^{r+b+s}-\frac12 x^{r+b}yx^{s+1} -\frac12 x^{r+1}yx^{b+s} +\frac12 x^{r+1+b}yx^{s} -\frac12 x^{r}yx^{b+s+1}\right)\nonumber\\
&&+\sum_{\genfrac{}{}{0pt}{2}{r, s\ge 0}{r+s=b-1}}\sum_{\genfrac{}{}{0pt}{2}{r', s'\ge 0}{r'+s'=a-1}} \left(\frac12 x^{r'+s}yx^{r+s'+2}-\frac12 x^{r+s'+2}yx^{r'+s}\right)\,,
\end{eqnarray*}  
from which we get that
\begin{equation*}
\begin{aligned}
y\{yx^b, x^a\}&= \sum_{\genfrac{}{}{0pt}{2}{r, s\ge 0}{r+s=a-1}}( -yx^{r+b+s}-\frac12 yx^{r+b}yx^{s+1} -\frac12 yx^{s+1}yx^{b+r} +\frac12 yx^{r+1+b}yx^{s} -\frac12 yx^{s}yx^{b+r+1})\nonumber\\
&+\sum_{\genfrac{}{}{0pt}{2}{r, s\ge 0}{r+s=b-1}}\sum_{\genfrac{}{}{0pt}{2}{r', s'\ge 0}{r'+s'=a-1}} \frac12 [yx^{r'+s},yx^{r+s'+2}]\,,
\end{aligned}
\end{equation*}  
which, modulo commutators, gives
\begin{equation}\label{Term2}
y\{yx^b, x^a\}=-ayx^{a+b-1}-\sum_{\genfrac{}{}{0pt}{2}{r, s\ge 0}{r+s=a-1}}yx^{s+1}yx^{r+b}\mod [A,A]\,.
\end{equation}  
Finally, for $a=1\mod m$:
\begin{eqnarray*}
\{yx^a, y\}\,&=& -\frac12 x^ay^2+\frac12 y^2x^a\\
&&+\sum_{\genfrac{}{}{0pt}{2}{r, s\ge 0}{r+s=a-1}} \left( x^syx^{r}+\frac12 yx^{s+1}yx^{r} +\frac12 x^syx^{r+1}y -\frac12 yx^syx^{r+1} +\frac12 x^{s+1}yx^{r}y \right)\,,
\end{eqnarray*}
and therefore
\begin{equation*}
\begin{aligned}
\{yx^a, y\}x^b\,&= \frac12 [y^2x^b,x^a]\\
&+\sum_{\genfrac{}{}{0pt}{2}{r, s\ge 0}{r+s=a-1}} ( x^syx^{r+b}+\frac12 yx^{s+1}yx^{r+b} +\frac12 x^syx^{r+1}yx^b -\frac12 yx^syx^{r+1+b} +\frac12 x^{s+1}yx^{r}yx^b)\,.
\end{aligned}
\end{equation*}
Modulo commutators, this gives
\begin{equation*}
\sum_{\genfrac{}{}{0pt}{2}{r, s\ge 0}{r+s=a-1}} \left( yx^{r+s+b}+\frac12 yx^{s+1}yx^{r+b} +\frac12 yx^{r+1}yx^{s+b} -\frac12 yx^syx^{r+1+b} +\frac12 yx^{r}yx^{s+1+b} \right)\,.
\end{equation*}
Therefore, 
\begin{equation}\label{Term1}
\{yx^a, y\}x^b\,=ayx^{a+b-1}+\sum_{\genfrac{}{}{0pt}{2}{r, s\ge 0}{r+s=a-1}} yx^{s+1}yx^{r+b}\mod [A,A]\,.
\end{equation}
Now, using \eqref{Term2}, \eqref{Term1}, we obtain modulo commutators: 
\begin{equation*}
\{yx^b, yx^c\}=\{yx^b, y\}x^c+y\{yx^b, x^c\}=(b-c)yx^{b+c+1}+\sum_{\genfrac{}{}{0pt}{2}{r, s\ge 0}{r+s=b-1}} yx^{s+1}yx^{r+c}-\sum_{\genfrac{}{}{0pt}{2}{r, s\ge 0}{r+s=c-1}} yx^{s+1}yx^{r+b}\,,
\end{equation*}
which gives \eqref{Eq:cycdgg}. This finishes the proof of Proposition \ref{Prop:dBrcyclic}. 
\end{proof}

For the tadpole quiver, the brackets in \eqref{tadida}--\eqref{tadide} look entirely similar (with small sign differences), and all the above proofs carry over with $E_r=E_0=e_0\otimes e_0$ for all $r$, leading to
\begin{prop}
\label{Prop:dBrTadapp}
For any $a, b\ge 0$ we have
  \begin{subequations}
       \begin{align}
\{x^a, x^b\}&=\{y^a, y^b\}=\{(xy)^a, (xy)^b\}=\{z^a, z^b\}=0\,,\quad\text{where}\ z=y+x^{-1}\,,\\
\{x^a, yx^b\}&=ax^{a+b-1}+ayx^{a+b} \mod{[A,A]}\,,\label{Eq:tadfg2}\\
\{yx^a, yx^b\}&=(a-b) yx^{a+b-1}+ \sum \limits_{t=1}^{a} yx^tyx^{a+b-t}- \sum \limits_{t=1}^{b} yx^tyx^{a+b-t} \mod{[A,A]}\,.\label{Eq:tadgga}
   \end{align}
  \end{subequations}
\end{prop} 
Finally, replacing $y$ by $z-x^{-1}$ in the last two formulas, one gets after a simple rearrangment the formulas from Propositions \ref{Prop:dBrTad}.
\qed

\section{Hamiltonian dynamics} \label{Ann:dynamics} 
In this section we prove Propositions \ref{dynam}. We start with a simple formula, immediate from \eqref{derr}: for any $a,b\in A$,
\begin{equation*}
\{\tr a,\, b_{ij}\}=\{a, b\}_{ij}\,.
\end{equation*}
It is convenient to rewrite this relation in matrix form, using the notation $\mathcal{X}(a)=(a_{ij})$ for the matrix-valued function on the representation space, associated to $a\in A$:
\begin{equation}\label{trx}
\{\tr \mathcal{X}(a), \mathcal{X}(b)\}=\mathcal{X}(\{a, b\})\,.
\end{equation}

Let now $\bar{Q}$ be the doubled framed quiver as in Section \ref{cyclic}, with the double bracket defined by the bivector \eqref{Eq:PP}, and $\{-,-\}$ be the associated bracket $\{-,-\}=m\circ \dgal{-,-}$. Recall that $\{y^k, y\}=0$ for all $k$.
\begin{lem}
For any $k\in m\N$,
\begin{equation*} 
\{y^k, v\}=\{y^k, w\}=0\,,\qquad \{y^k, x\}=-ky^{k-1}-ky^kx\,. 
\end{equation*}
\end{lem}
\begin{proof}
From \eqref{a<b} we have 
\begin{eqnarray*}
\dgal{y_{i},v}&=&-\frac12 (y_{m-1}\otimes v)\delta_{i, m-1}+\frac12 (e_0\otimes y_0v)\delta_{i,0}\,, \\
\dgal{y_{i},w}&=&\frac12 (wy_{m-1}\otimes e_0)\delta_{i, m-1}-\frac12 (w\otimes y_0)\delta_{i,0}\,.
\end{eqnarray*}
Using this, one preforms a calculation as in Appendix \ref{Ann:brackets} to find that 
\begin{equation*}
\dgal{v, y^k}=\sum_{r+s=k-1} y^r\dgal{v,y}y^s=\sum_{r+s=k-1} \left(\frac12 y^rv\otimes e_0y^{s+1}-\frac12 y^{r+1}v\otimes e_0y^s\right)\,, 
\end{equation*}
and $\{y^k, v\}=0$ as a result. The relation $\{y^k, w\}=0$ is checked similarly. The last formula in the lemma is analogous to \eqref{xay}. 
\end{proof}
Now consider a representation space $\Rep(\CC \bar{Q}, \aalpha)$ for a dimension vector $\aalpha=(1, \alpha)\in\N^{m+1}$, and let as before $V,W, X, Y$ be the matrices representing the arrows $v, w$ and $x=\sum_i x_i$, $y=\sum_i y_i$. The space $\Rep(\CC \bar{Q}, \aalpha)$ is equipped with an anti-symmetric biderivation $\{-,-\}$ defined by \eqref{derr}. 
We put $H_k=\frac1k\tr Y^k$, where $Y$ is the matrix representing $y=\sum_{i\in I} y_i$ and $k\in m\N$. Then $\{H_k, -\}$ defines a derivation (vector field) on the representation space. Since $\{H_k, H_l\}=0$ for all $k,l$ and $\{ -,-\}$ is a Loday bracket (see \eqref{loday}), the vector fields associated to $H_k$ with different $k$ pairwise commute. Let $\frac{d}{dt_j}$ denote the vector field corresponding to $H_{jm}$. Using the above lemma and formula \eqref{trx}, we obtain: 
\begin{equation*}
\frac{d}{dt_j}V=\frac{d}{dt_j}W=0\,,\qquad \frac{d}{dt_j}Y=0\,,\quad \frac{d}{dt_j}X=-Y^{jm-1}-Y^{jm}X\,.
\end{equation*}
Integrating the last equation with constant $Y=Y(0)$, we find that 
\begin{equation*}
X(t_j)=e^{-t_jY^{jm}}X(0)+Y^{-1}(e^{-t_jY^{jm}}-1)\,.
\end{equation*}
This formula is well-defined for all $Y$ because the function $z^{-1}(e^{-t_kz^k}-1)$ is analytic in $z$. 
By superposition, for $\mathbf{t}=(t_1, t_{2}, \dots)$ we have:
\begin{equation*}
X(\mathbf{t})=e^{-\sum_{j\ge 1} t_{j}Y^{jm}}\,X(0)+Y^{-1}\left(e^{-\sum_{j\ge 1}t_jY^{jm}}-1\right)\,.
\end{equation*}
Note that $1+YX=e^{-\sum_{j\ge 1} t_{j}Y^{jm}}(1+YX(0))$, so the matrices $\Id_n+Y_sX_s$, $s\in I$ remain invertible for all times.

If instead of $H_k$ one considers $G_k:=\frac1k \tr Z^k$, $Z=Y+X^{-1}$, then the brackets look the same, apart from the bracket between $z$ and $x$ \eqref{xz}. As a result, if $t_j$ denotes the time flow for $G_{jm}$, we obtain:
 \begin{equation*}
\frac{d}{dt_j}V=\frac{d}{dt_j}W=0\,,\qquad \frac{d}{dt_j}Z=0\,,\quad \frac{d}{dt_j}X=-Z^{jm}X\,.
\end{equation*}
This leads to 
\begin{equation*}
X(\mathbf{t})=e^{-\sum_{j\ge 1} t_{j}Z^{jm}}\,X(0)\,,\quad Z=Z(0)\,.
\end{equation*}
In both cases, it is easy to see that the dynamics preserves the moment map equations. Therefore, the associated Hamiltonian flows are complete.
\qed

\section{Poisson isomorphism between $\mathcal C_{n,t}^0$ and $\mathcal C_{n, q}^0(m)$}\label{sums}

In this section we prove Proposition \ref{IsoPssCyc}. The map $\xi:\, \mathcal C_{n,t}^0 \to \mathcal{C}_{n,q}^0(m)$ is described in Proposition \ref{isocm}, namely:
\begin{equation*}
X_i=\Id_n\,,\quad Z_i=t_iB\ (i=0\dots m-2)\,,\quad X_{m-1}=A\,,\ Z_{m-1}=t_{m-1}A^{-1}B\,.
\end{equation*}
We need to check that  
$\xi^*\{f, g\}=\{\xi^*f,\xi^*g\}$ for any two functions on $\mathcal{C}_{n,q}^0(m)$. 
Consider the functions $f_\alpha:=\tr(X^{\alpha m})$ and $g_\beta:=\tr (Z X^{1+\beta m})$. Expressing them in terms of $A, B$ we obtain:
\begin{subequations}
\begin{align}
 \xi^*f_\alpha  = & \sum_{i\in I} \tr X_{i+1}\dots X_{i+\alpha m}=m\tr A^\alpha\,,\\
\xi^*g_\beta= & \sum_{i\in I} \tr Z_{i}X_{i}X_{i+1}\dots X_{i+\beta m}=\tau\tr BA^{\beta}\,,\quad\text{where}\  \tau:= \sum_{i\in I} t_i\,.
\end{align}
\end{subequations}
This implies that $f_1, \dots, f_n, g_1, \dots, g_n$ can be used as local coordinates near a generic point of $\mathcal{C}_{n,q}^0(m)$. Therefore, it is sufficient to check that the brackets behave well just for these functions. 

Let us rewrite the formulas from Proposition \ref{Prop:dBrcyclic} in the algebra $A'$ with inverted $x$. We use $z:=y+x^{-1}$ so that $zx=1+yx$, and rearrange \eqref{Eq:cycfg}--\eqref{Eq:cycdgg} as follows:
 \begin{subequations}
       \begin{align}
\{x^a, zx^{b}\}&=azx^{a+b} \mod{[A', A']}\,,\label{Eq:cycfgap}\\
\{zx^{b}, zx^{c}\}&=\sum \limits_{t=1}^{b} zx^{t}zx^{b+c-t}- \sum \limits_{t=1}^{c} zx^{t}zx^{b+c-t} \mod{[A', A']}\,.\label{Eq:cycdggap}
   \end{align}
  \end{subequations}
Recall that we also have $\{x^a, x^b\}=0$ for all $a,b$.  
Substituting $a=\alpha m$, $b=1+\beta m$, $c=1+\gamma m$ and taking traces using \eqref{trace} leads to:
 \begin{subequations}
       \begin{align}
\{f_\alpha, f_\beta\}&=0\,,\quad \{f_\alpha, g_\beta\}=\alpha m\, g_{\alpha+\beta} \,,\label{Eq:cycfgapt}\\
\{g_{\beta}, g_{\gamma}\}&=\sum \limits_{r=0}^{\beta m} h_{r, (\beta+\gamma)m-r} - \sum \limits_{r=0}^{\gamma m} h_{r, (\beta+\gamma)m-r} \,.\label{Eq:cycdggapt1}
\end{align}
  \end{subequations}  
Here we used the notation $h_{r,s}:=\tr Z X^{1+r} Z X^{1+s}$. Using that $h_{r,s}=h_{s,r}$ and assuming $\beta<\gamma$, the last relation can be rearranged to
\begin{equation}
\{g_{\beta}, g_{\gamma}\}=\sum \limits_{r=\beta m}^{\gamma m-1} h_{r, (\beta+\gamma)m-r} \,.\label{Eq:cycdggapt}
\end{equation}
On the other hand, the Poisson bracket on $\mathcal C_{n,t}^0$ satisfies \eqref{old}, which in terms of $A,B$ reads:
 \begin{subequations}
       \begin{align}
\{\tr A^a, \tr A^b\}&=0\,,\quad \{\tr A^a, \tr BA^b\}= a\, \tr BA^{a+b}\,,\\
\{\tr BA^b, \tr BA^c\}&=\sum \limits_{r=0}^b \tr BA^rBA^{b+c-r}-\sum \limits_{r=0}^c \tr BA^rBA^{b+c-r}=\sum_{r=b}^{c-1}\tr BA^rBA^{b+c-r}\,.
\end{align}
  \end{subequations}  
It remains to check that these formulas agree with those obtained from \eqref{Eq:cycfgapt} and \eqref{Eq:cycdggapt}  by replacing $f_\alpha$, $g_\beta$ with their pull-backs $\xi^*f_\alpha=m\tr A^\alpha$ and $\xi^*g_\beta=\tau\tr BA^\beta$. For the first two relations this is obvious. For the third relation, we need to show that for $\beta<\gamma$
\begin{equation}\label{hid}
\sum \limits_{r=\beta m}^{\gamma m-1} \xi^*h_{r, (\beta+\gamma)m-r}=\tau^2\sum \limits_{p=\beta}^{\gamma-1} \tr BA^pBA^{\beta+\gamma-p}\,.
\end{equation}     
We have $\xi^*h_{r,s}=\tr XZX^rXZX^s$, with
\begin{equation*}
XZX^rXZX^s=\sum_{i\in I}  X_iZ_iX_i\dots X_{i+r-1}X_{i+r}Z_{i+r}X_{i+r}\dots X_{i+r+s-1}\,.
\end{equation*}
By expressing this in terms of $A,B$ and taking traces we obtain:
\begin{equation*}
\xi^*h_{r,s}= \sum_{i\in I} t_{i}t_{i+r}\tr BA^{\phi(i, i+r)}BA^{\phi(i+r, i+r+s)}\,,
\end{equation*}
where $\phi(i,j)=\left[\frac{j}{m}\right]-\left[\frac{i}{m}\right]$ counts the number of integers between $i$ and $j-1$ congruent to $-1$ modulo $m$. Here and below the indices in $t_i$ are always treated modulo $m$, i.e., $t_{i+m}=t_i$. Therefore, if $r=j+pm$ then
\begin{equation*}
\xi^*h_{r, (\beta +\gamma)m -r}= \sum_{i\in I} t_{i}t_{i+j}\tr BA^{p+\phi(i, i+j)}BA^{\beta+\gamma- p- \phi(i, i+j)}\,.
\end{equation*}
Note that for $i, j\in I$, $\phi(i, i+j)$ equals $0$ or $1$, depending on whether $i+j\le m-1$ or not. We can use this in the l.h.s. of \eqref{hid}, replacing the summation over $r$ with the summation over $p\in [\beta, \gamma-1]$ and $j\in I$, which leads to
\begin{equation*}
 \sum_{\genfrac{}{}{0pt}{2}{i, j\in I}{i+j\le m-1}}
 t_{i}t_{i+j} \sum \limits_{p=\beta}^{\gamma-1} \tr BA^{p}BA^{\beta+\gamma-p}
+ \sum_{\genfrac{}{}{0pt}{2}{i, j\in I}{i+j> m-1}}
 t_{i}t_{i+j} \sum \limits_{p=\beta}^{\gamma-1} \tr BA^{p+1}BA^{\beta+\gamma-p-1}\,.
\end{equation*}
It is easy to see that $\sum \limits_{p=\beta}^{\gamma-1} \tr BA^{p+1}BA^{\beta+\gamma-p-1}=\sum \limits_{p=\beta}^{\gamma-1} \tr BA^{p}BA^{\beta+\gamma-p}$, therefore, we arrive at
\begin{equation*}
\sum_{i, j\in I} t_{i}t_{i+j}\sum \limits_{p=\beta}^{\gamma-1} \tr BA^{p}BA^{\beta+\gamma-p}\,.
\end{equation*}
It remains to notice that $\sum_{i, j\in I} t_{i}t_{i+j}=(\sum_{i\in I} t_{i})^2=\tau^2$, so the obtained expression coincides with the r.h.s. of the relation \eqref{hid}. This finishes the proof of \eqref{hid} and of the Proposition \ref{IsoPssCyc}.
\qed


\end{document}